\documentclass[11pt,a4paper,reqno]{amsart}
\usepackage{amsmath}
\usepackage{amssymb}
\usepackage{euscript}
\usepackage[all]{xy}
\usepackage{graphicx}
\usepackage{xcolor}
\usepackage[colorlinks=true,linktocpage=true,pagebackref=false, citecolor=black,linkcolor=black]{hyperref}

\newtheorem{thm}{Theorem}[section]

\newtheorem{cor}[thm]{Corollary}
\newtheorem{lem}[thm]{Lemma}

\theoremstyle{definition}

\theoremstyle{remark}

\newtheorem{remarks}[thm]{Remarks}

\newcommand{\K}{{\mathbb K}}

\newcommand{\R}{{\mathbb R}}
      
\newcommand{\C}{{\mathbb C}}
\newcommand{\sph}{{\mathbb S}}

\newcommand{\gtm}{{\mathfrak m}}     
\newcommand{\gta}{{\mathfrak a}}

\newcommand{\an}{{\EuScript O}}
\newcommand{\mer}{{\EuScript M}}

\newcommand{\Ii}{{\EuScript I}}
\newcommand{\Zz}{{\EuScript Z}}
\newcommand{\Ss}{{\EuScript S}} 
\newcommand{\Jj}{{\EuScript J}}
\newcommand{\Ff}{{\EuScript F}}
\newcommand{\Gg}{{\EuScript G}}
\newcommand{\Bb}{{\EuScript B}}

\newcommand{\Tt}{{\EuScript T}}

\newcommand{\Reg}{\operatorname{Reg}}
\newcommand{\Sing}{\operatorname{Sing}}
\newcommand{\im}{\operatorname{im}}           
\newcommand{\Int}{\operatorname{Int}}          
\newcommand{\dist}{\operatorname{dist}}
\newcommand{\supp}{\operatorname{supp}}
\newcommand{\cl}{\operatorname{Cl}}

\newcommand{\id}{\operatorname{id}}

\newcommand{\x}{{\tt x}}
\newcommand{\y}{{\tt y}}     
\newcommand{\z}{{\tt z}}     
\newcommand{\s}{{\tt s}}

\newcommand{\Sos}{{\textstyle\sum}} 
\newcommand{\veps}{\varepsilon}
\newcommand{\ol}{\overline}

\begin{document}

\title[Positive semidefinite analytic functions on real analytic surfaces]{Positive semidefinite analytic functions\\ on real analytic surfaces}

\author{Jos\'e F. Fernando}
\address{Departamento de \'Algebra, Geometr\'\i a y Topolog\'\i a, Facultad de Ciencias Matem\'aticas, Universidad Complutense de Madrid, 28040 MADRID (SPAIN)}
\curraddr{}
\email{josefer@mat.ucm.es}
\thanks{Author supported by Spanish STRANO MTM2017-82105-P and Grupos UCM 910444}

\date{14/06/2021}

\begin{abstract} 
Let $X\subset\R^n$ be a (global) real analytic surface. Then every positive semidefinite meromorphic function on $X$ is a sum of $10$ squares of meromorphic functions on $X$. As a consequence, we provide a real Nullstellensatz for (global) real analytic surfaces.
\end{abstract}
\maketitle

\section{Introduction}\label{s1}

The famous Hilbert's 17th Problem asks whether positive semidefinite functions can be represented as sums of squares and in that case how many squares are needed. The two parts of this question are distinguished as the \em qualitative \em and the \em quantitative \em aspects of the problem. Recall that the \em Pythagoras number \em of a ring $A$ is the smallest integer $p(A)\ge1$ such that every sum of squares of $A$ is a sum of $p$ squares or infinity if such an integer does not exist.

The specialists have studied both problems for different types of functions: polynomial \cite{ch80,sch3}, regular \cite{sch1,sch4}, rational \cite{ea,cep,pf} and \cite[\S6]{bcr}, Nash \cite[\S8.5]{bcr}, regulous \cite{fmq}, smooth \cite{by,bbcp,gl}, analytic, meromorphic, \ldots and found full or partial solutions in many cases. Global analytic and meromorphic functions remain as the most defying types. 

\subsection{Analytic functions.}
Let us recall the state of the art in the local and global analytic settings. The qualitative and quantitative problems have been widely approached for analytic set germs in both the analytic and meromorphic contexts. Denote the ring of convergent power series with $\R\{\x\}:=\R\{\x_1,\ldots,\x_n\}$. Given an analytic set germ $X$ (at the origin), denote the ideal of all convergent power series that vanish identically on $X$ with $\Jj(X)$. The analytic ring of $X$ is $\an(X):=\R\{\x\}/\Jj(X)$.

In \cite{fe2,frs1} we proved that if $X$ is an analytic set germ of dimension $\geq3$, both the qualitative and quantitative questions have negative answers, that is, there are analytic function germs on $X$ that are positive semidefinite but they are not sums of squares of analytic function germs on $X$ and $p(\an(X))=+\infty$. Thus, one is only concerned about dimensions $1$ and $2$. In \cite[Thm.3.9]{sch2}, Scheiderer characterized the analytic curve germs $X$ with the property that every positive semidefinite analytic function is a sum of squares (and in fact one square): \em $X$ is a (finite) union of non-singular independent branches\em, or equivalently, \em the ring of analytic function germs on $X$ is
$$
\an(X)\cong\R\{\x_1,\ldots,\x_n\}/(\x_i\x_j:\ 1\leq i<j\leq n).
$$\em

In addition, $1\leq p(\an(X))\leq{\rm mult}(X)$ for each analytic curve germ $X$, where ${\rm mult}(X)$ is the multiplicity of $X$ (see \cite{fq,q} and \cite{or} for further results). In \cite{fe3,fe4,fe5,fe7,fr1,rz99} we provide the complete list of the analytic surface germs $X$ of embedding dimension $\leq3$ such that every positive semidefinite analytic function germ on $X$ is a sum of squares of analytic function germs on $X$ (in fact only $2$ squares are needed). We refer the reader to \cite{fe17} when dealing with more general residue fields (for instance, the rational numbers). In addition, we exhibit in \cite{fe4} families of analytic surface germs $X$ of arbitrary embedding dimension such that every positive semidefinite analytic function germ on $X$ is a sum of squares of analytic function germs on $X$ (in fact only $2$ squares are needed). Concerning Pythagoras numbers, we proved in \cite{fe1} that $2\leq p(\an(X))\leq2{\rm mult}(X)$ for each analytic surface germ $X$, where ${\rm mult}(X)$ is the multiplicity of $X$. We refer the reader to \cite{frs2} when dealing with more general residue fields (for instance, the rational numbers). The relationship between the Pythagoras number of the ring of analytic functions germs on an analytic set germ $X$ and the Pythagoras numbers of the rings of analytic function germs on analytic curve germs $Y$ contained in $X$ is studied in \cite{fr2}.

Let us now approach the global analytic setting. By Whitney's immersion theorem \cite[Thm.2.15.7]{n4} a real analytic manifold $M$ of dimension $d$ can be embedded in $\R^{2d+1}$. As open subsets of $\R^n$ are real analytic manifolds, when dealing with global analytic subsets of an open subset $\Omega$ of $\R^n$ (that we call $C$-analytic subsets of $\Omega$ in honor of Cartan \cite{c}), we may always assume (by Whitney's immersion theorem and Cartan's Theorem B) that they are $C$-analytic subsets of certain $\R^m$. Thus, for the sake of simplicity, we will only analyze what happens with $C$-analytic subsets of $\R^n$ (and this includes `a fortiori' the general case of $C$-analytic subsets of real analytic manifolds via analytic tubular neighborhoods, Whitney's immersion theorem and Cartan's Theorem B). We will also deal with $C$-analytic spaces, which are introduced in \S\ref{complexification}. 

Denote the sheaf of germs of analytic functions on $\R^n$ with $\an_{\R^n}$ and its ring of global sections with $\an(\R^n):=H^0(\R^n,\an_{\R^n})$, that is, the ring of (global) analytic functions on $\R^n$. If $X\subset\R^n$ is a $C$-analytic set, that is, it is the zero-set of a real analytic function $f\in\an(\R^n)$, we associate the ideal $\Jj(X)$ of analytic functions on $\R^n$ that vanish identically on $X$ to $X$. We consider the sheaf of ideals $\Jj_X:=\Jj(X)\an_{\R^n}$ and the sheaf of (quotient) rings $\an_X:=\an_{\R^n}/\Jj_X$, whose ring of global sections $\an(X):=H^0(X,\an_X)=\an(X)/\Jj(X)$ is the ring of (global) analytic functions on $X$. Recall that $X$ is \em coherent \em if for each $x\in X$ the ideal $\Jj_{X,x}$ coincides with the ideal of analytic function germs on $\R^n$ that vanishes identically on the analytic set germ $X_x$.

Again the dimension $3$ determines a substantial change of behavior. If $X$ is a $C$-analytic set of dimension $\geq 3$, both the qualitative and quantitative questions have negative answers, that is, there are (global) analytic functions on $X$ that are positive semidefinite but they are not sums of squares of analytic functions on $X$ and $p(\an(X))=+\infty$ (see \cite{fe2,fe7}). Thus, again one is only concerned about dimensions $1$ and $2$. The one dimensional case is fully approached in \cite{abfr2}. Every positive semidefinite analytic function on an analytic curve $X$ is a sum of squares if and only if for each $x\in X$ the analytic curve germ $X_x$ has the corresponding local property (and in fact only $2$ squares are needed). Concerning Pythagoras numbers, if $X$ is an analytic curve, then
$$
p(\an(X))=\sup\{p(\an(X_x)):\ x\in X\}+\veps
$$ 
where $\veps$ is either $0$ or $1$. The $2$ dimensional case is more delicate. If $X$ is a non-singular analytic surface, then every positive semidefinite analytic function on $X$ is a sum of squares of analytic functions on $X$ (and at most $3$ squares are needed, \cite{br,jw1}). In fact, if $X$ is in addition connected, Jaworski \cite[Cor.2]{jw1} proved
$$
p(\an(X))=\begin{cases}
2&\text{if $X$ is non-compact,}\\
3&\text{if $X$ is compact.}
\end{cases}
$$
Analytic surfaces with singular points are approached in \cite{fe7}. In particular, we prove that if $X$ is an analytic surface such that every positive semidefinite analytic function on $X$ is a sum of squares of analytic functions on $X$, then $X$ is coherent and for each $x\in X$ the analytic surface germ $X_x$ has the corresponding local property. If the local embedding dimension of $X$ is $\leq3$, we prove that the converse implication holds (and in fact only $6$ squares are needed). As far as we know, the quantitative problem in the $C$-analytic $2$-dimensional case has not been studied in detail if $X$ has singular points and remains as an open problem.

\subsection{Meromorphic functions.}
If $X$ is an analytic set germ at the origin of $\R^n$, we denote the total ring of fractions of the ring $\an(X)$ of analytic function germs on $X$ with $\mer(X)$. Risler and Ruiz provided a full positive answer to the qualitative problem and proved in \cite{ri,rz00} that every positive semidefinite meromorphic function germ on an analytic set germ $X$ is a (finite) sum of squares of meromorphic function germs on $X$. The quantitative problem in the local setting has been completely solved by Hu \cite{hu} and Benoist \cite{b}. If $n=1$, it is straightforward that $p(\mer(\R_0))=1$, whereas in \cite{ch80} it was shown $p(\mer(\R^2_0))=2$. Jaworski proved in \cite{jw3} that $4\leq p(\mer(\R^3_0))\leq8$, whereas Hu showed in \cite[\S5]{hu} that $p(\mer(\R^n_0))\leq 2^n$ and $p(\mer(\R^3_0))=4$. Benoist improved Hu's bound obtaining that $p(\mer(\R^n_0))\leq 2^{n-1}$ for each $n\geq1$. Let us sketch how $p(\mer(X))\leq 2^d$ for each analytic set germ $X$ of dimension $d$. By R\"uckert's parameterization theorem \cite[Prop.3.4]{rz99} $X$ embeds in $\R^{d+1}_0$ as an analytic hypersurface germ through a birational model. Then, each sum of squares of meromorphic function germs on $X$ is the restriction of one on $\R^{d+1}_0$, which is a sum of $2^d$ squares of meromorphic function germs by \cite{b}. This sum of $2^d$ squares restricts well to $X$, because the analytic equation of $X$ in $\R^{d+1}_0$ is real, so it can be factored out from the poles of the $2^d$ addends. 

If $X$ is an analytic curve germ and $\widehat{X}$ is its normalization, it holds $\mer(X)\cong\mer(\widehat{X})\cong\bigoplus_{i=1}^r\mer(\R_0)$ (where $r$ is the number of branches of $X$), so $p(\mer(X))=1$. In \cite{abfr1} we proved that $p(\mer(X))\leq 4$ for each analytic surface germ $X$ and control in addition the zero-set of poles of the representations as sums of squares. 

\begin{thm}[{\cite[Thm.1.3]{abfr1}}]\label{localcase}
Let $X$ be an analytic surface germ and $f\in\an(X)$ a positive semidefinite analytic
function germ. Then there are analytic function germs $h_0,h_1,h_2,h_3,h_4\in\an(X)$ such that $h_0^2f=h_1^2+h_2^2+h_3^2+h_4^2$ and $h_0$ is a sum of squares in $\an(X)$ with $\{h_0=0\}\subset\{f=0\}$.
\end{thm}

In the global meromorphic setting we have less information. Recall that the ring of meromorphic functions $\mer(X)$ on a $C$-analytic set $X$ is the ring of global sections of the sheaf $\mer_X$, whose stalks $\mer_{X,x}$ are the total rings of fractions of $\an_{X,x}$ for each $x\in X$ (see \cite[VIII.B]{gr2}). Each $C$-analytic set $X\subset\R^n$ has a \em complexification\em, that is, a complex analytic subset $Y$ of an open neighborhood $\Omega\subset\C^n$ of $\R^n$ such that $Y\cap\R^n=X$ and for each $x\in X$ the ideal $\Jj_{X,x}\otimes\C$ coincides with the ideal $\Jj(Y_x)$ of $\an_{\C^n,x}$ of holomorphic function germs at $x$ that vanish identically on the analytic germ $Y_x$.

As $X\subset\R^n$ is a $C$-analytic set, it has by \cite{c,wb} a system of invariant (under conjugation) open Stein neighborhoods in a {\em complexification} $Y\subset\C^n$ of $X$. By \cite[VIII.B.Cor.10]{gr2} we conclude that $\mer(X)$ is the total ring of fractions of $\an(X)$.

The qualitative and the quantitative problems are still open for each $C$-analytic set $X$ of dimension $d\geq3$. There exist only some partial results for $\R^n$ if $n\geq3$. For instance, Bochnak-Kucharz-Shiota proved in \cite{bks} that every analytic function on $\R^n$ whose zero-set is a discrete set is a sum of $2^n+n+1$ squares of meromorphic functions on $\R^n$. In addition, Ruiz showed in \cite{rz1} that if the zero-set of a positive semidefinite analytic function $f$ on a $C$-analytic set $X$ is compact, then $f$ is a sum of squares of meromorphic functions on $X$. Unfortunately, there is no bound concerning the number of squares involved. In \cite{jw2} Jaworski showed that if $f$ is a positive semidefinite analytic function on $X$ whose zero-set is discrete outside a compact set, then $f$ is a sum of squares of meromorphic functions on $X$.

In \cite{abfr3} we propose to involve also infinite (convergent) sums of squares of analytic functions. We showed that if the connected components of the zero-set of a positive semidefinite analytic function $f$ on a real analytic manifold $M$ are compact, then $f$ is an infinite sum of squares of meromorphic functions on $M$. If there exists a bound for the number of squares needed to represent $f$ around the (compact) connected components of $\{f=0\}$, then the previous sum of squares of meromorphic functions is finite. This requires to control the zero-set of the set of poles of the representation of $f$ as a sum of squares of meromorphic functions around the connected components of $\{f=0\}$ (see \cite[Thm.1.5]{abfr3}). In addition, we proved in \cite[Prop.1.11]{abfr3} that a positive answer to the qualitative problem for $\mer(\R^n)$ implies the finiteness of the Pythagoras number of $\mer(\R^n)$. In \cite{abf2,abf6,fe8} we showed that the obstruction for a positive semidefinite analytic function $f$ on $\R^n$ to be a (maybe infinite) sum of squares of meromorphic functions on $\R^n$ concentrates around the zero-sets of the invariant (under complex conjugation) irreducible factors of $f$ (of odd multiplicity) whose zero-sets have dimension between $1$ and $n-2$. 

Thus, if $n\leq2$, such obstructions do not appear and let us comment what is known for $C$-analytic curves and surfaces. By \cite[Thm.3.9]{as} we know that if $Y$ is a complexification of $X$ and $(\widehat{Y},\pi)$ is the normalization of $Y$, then $\mer(Y)\cong\mer(\widehat{Y})$. If in addition $X$ is coherent, then $\widehat{X}:=\pi^{-1}(X)$ is the real part space of $\widehat{Y}$ (so it is a normal real analytic space) and $\mer(X)\cong\mer(\widehat{X})$. Thus, if $X$ is coherent, we may assume that it is normal in order to approach the qualitative and the quantitative questions. However, as the bounded local embedding dimension of the normalization is not guaranteed in the $2$-dimensional case, we have to deal with normal $C$-analytic spaces of dimension $2$ instead of normal $C$-analytic surfaces of some $\R^n$. 

If $X$ is a normal analytic curve, its connected components are analytically diffeomorphic to lines and circumferences. If $\{X_k\}_k$ is the collection of the connected components of $X$, then
$$
\an(X)\cong\prod_k\an(X_k)\quad\text{and}\quad\mer(X)\cong\prod_k\mer(X_k).
$$
Thus, every positive semidefinite analytic function of $X$ is a sum of squares of analytic functions on $X$ (and $2$ squares are enough). In fact, $p(\mer(X))=\max_k\{p(\mer(X_k))\}\leq2$ and by \cite[Cor.1]{jw1}
$$
p(\mer(X_k))=\begin{cases}
1&\text{if $X_k$ is unbounded,}\\
2&\text{if $X_k$ is compact.}
\end{cases}
$$

If $X$ is a normal $C$-analytic surface, we have the following result (see also \cite{adr}). 
\begin{thm}[{\cite[Thm.1.4]{abfr1}}]
Let $X$ be a normal $C$-analytic surface and let $f:X\to\R$ be a positive semidefinite analytic function. Then, there exist analytic functions $g,f_1,f_2,f_3,f_4,f_5\in\an(X)$ such that $g^2f=f_1^2+f_2^2+f_3^2+f_4^2+f_5^2$ and $g$ is a sum of squares whose zero-set $\{g=0\}$ is a discrete subset of $\{f=0\}$.
\end{thm}

\subsection{Main result}\label{mainr}
Let us analyze what happens with general $C$-analytic surfaces. Denote the set of points $x$ of a $C$-analytic set $X\subset\R^n$ at which $\an_{X,x}$ is not a normal ring with $B(X)$. By \cite[Ch.VI.Thm.5]{n1} $B(X)$ is a $C$-analytic subset of $X$. In this work we prove the following result, whose proof works \em verbatim \em for $C$-analytic spaces of dimension $2$ (see \S\ref{complexification}).

\begin{thm}[Hilbert's 17th problem]\label{h17}
Let $X$ be a $C$-analytic surface and $f:X\to\R$ a positive semidefinite analytic function. Then there exist analytic functions $g,f_1,\ldots,f_{10}\in\an(X)$ such that $g^2f=\sum_{i=1}^{10}f_i^2$ and $\{g=0\}\subset\{f=0\}\cup B(X)$. In addition,
$$
p(\mer(X))\leq\begin{cases}
5&\text{if $X$ is coherent,}\\
10&\text{if $X$ is non-coherent.}
\end{cases}
$$
\end{thm}

As we use the normalization of $(X,\an_X)$ to prove Theorem \ref{h17}, it seems difficult to get rid of the set $B(X)$ in the zero-set $\{g=0\}$ of the denominator $g$.

\subsection{Real Nullstellensatz}
In close relation to Hil\-bert's 17th problem there is a classical result in Real Geometry: \em the real Nullstellensatz \cite{st}\em. Given an ideal $\gta$ of a (commutative unital) ring $A$ we define its real radical ideal as
$$
\sqrt[r]{\gta}:=\Big\{f\in A:\ f^{2m}+\sum_{i=1}^ra_i^2\in\gta,\ a_i\in A, m,r\geq1\Big\}.
$$
If $X\subset\R^n$ is a $C$-analytic set and $\gta$ is an ideal of $\an(X)$, we define its \em saturation \em as $\widetilde{\gta}:=H^0(X,\gta\an_X)$. In addition, we denote $\Zz(\gta):=\{x\in X:\ f(x)=0\ \forall f\in\gta\}$. As a direct consequence of Theorem \ref{h17} and \cite{abf3}, we obtain the following real Nullstellensatz. 

\begin{cor}[Real Nullstellensatz]\label{null}
Let $X$ be a $C$-analytic set of dimension $\leq2$ and $\gta$ an ideal of $\an(X)$. Then $\Jj(\Zz(\gta))=\widetilde{\sqrt[r]{\gta}}$.
\end{cor}

The fact that the set $B(X)$ of non-normal points of $X$ appears in the set of poles of a representation of a positive semidefinite analytic function on a $C$-analytic surface $X$ as sums of squares of meromorphic functions makes it difficult, even in the coherent case, to obtain a weak Positivstellensatz in the sense of \cite{abf1}. 

\subsection*{Structure of the article}
The article is organized as follows. In Section \ref{s2} we present some preliminary concepts and results that will allow us to prove Theorem \ref{h17} in Section \ref{s3}. We will treat a `full approximation' result (Theorem \ref{fa}) with special care in Section \ref{s2}, which has interest on its own, concerning approximation of continuous functions on the real part space of a Stein space (endowed with an anti-involution) by invariant holomorphic functions defined in the whole Stein space.

\subsection*{Acknowledgements}
The author is very grateful to S. Schramm for a careful reading of the final version and for the suggestions to refine its redaction.

\section{Preliminaries}\label{s2}

In the following \em holomorphic \em refers to the complex case and \em analytic \em to the real case. For a further reading about complex analytic spaces we refer to \cite{gr2}, whereas we remit the reader to \cite{gmt,t} for the theory of real analytic spaces. We denote the elements of $\an(X):=H^0(X,\an_X)$ with capital letters if $(X,\an_X)$ is a complex analytic space and with lowercase letters if $(X,\an_X)$ is a real analytic space. 

A \em positive semidefinite (global) analytic function \em on a real analytic space $(X,\an_X)$ is an element $f\in\an(X)$ such that $f(x)\geq 0$ for each $x\in X$. We denote the set of all sums of (resp. $p$) squares of the ring $\an(X)$ of analytic functions on $X$ with $\Sos\an(X)^2$ (resp. $\Sos_p\an(X)^2$).

\subsection{General terminology}
Denote the coordinates in $\C^n$ with $z:=(z_1,\ldots,z_n)$ where $z_i:=x_i+\sqrt{-1}y_i$. Consider the conjugation $\ol{\,\cdot\,}:\C^n\to\C^n,\ z\mapsto\ol{z}:=(\ol{z_1},\ldots,\ol{z_n})$ of $\C^n$, whose set of fixed points is $\R^n$. A subset $A\subset\C^n$ is \em invariant \em if $\ol{A}=A$. Obviously, $A\cap\ol{A}$ is the biggest invariant subset of $A$. Let $\Omega\subset\C^n$ be an invariant open set and $F:\Omega\to\C$ a holomorphic function. We say that $F$ is \em invariant \em if $F(z)=\ol{F(\ol{z})}$ for each $z\in\Omega$. This implies that $F$ restricts to a real analytic function on $\Omega\cap\R^n$. Conversely, if $f$ is analytic on $\R^n$, it can be extended to an invariant holomorphic function $F$ on some invariant open neighborhood $\Omega$ of $\R^n$. If $F:\Omega\to\C$ is a holomorphic function and $\Omega$ is invariant, then
$$
\Re(F):\Omega\to\C,\ z\mapsto\tfrac{F(z)+\ol{F(\ol{z})}}{2}\quad\text{and}\quad\Im(F):\Omega\to\C,\ z\mapsto\tfrac{F(z)-\ol{F(\ol{z})}}{2\sqrt{-1}}
$$ 
are invariant holomorphic functions that satisfy $F=\Re(F)+\sqrt{-1}\,\Im(F)$. 

\subsection{Reduced analytic spaces \cite[\S I.1]{gmt}}
Let $\K:=\R$ or $\C$ and $(X,\an_X)$ be either a complex or real analytic space. Let $\Ff_X$ be the sheaf of $\K$-valued functions on $X$ and $\vartheta:\an_X\to\Ff_X$ the morphism of sheaves defined for each open set $U\subset X$ by 
$$
\vartheta_U(s):U\to\K,\ x\mapsto s(x),
$$ 
where $s(x)$ is the class of $s$ module the maximal ideal $\gtm_{X,x}$ of $\an_{X,x}$. Recall that $(X,\an_X)$ is \em reduced \em if $\vartheta$ is injective. Denote the image of $\an_X$ under $\vartheta$ with $\an_X^r$. The pair $(X,\an_X^r)$ is called the \em reduction \em of $(X,\an_X)$ and $(X,\an_X)$ is reduced if and only if $\an_X=\an_X^r$. The reduction is a covariant functor from the category of $\K$-analytic spaces to the one of reduced $\K$-analytic spaces.

\subsection{Anti-involution and complexifications \cite[\S II.4]{gmt}}
Let $(Y,\an_Y)$ be a complex analytic space. An \em anti-involution \em on $(Y,\an_Y)$ is a morphism $\sigma:(Y,\an_Y)\to(Y,\ol{\an}_Y)$ such that $\sigma^2=\id_Y$ and it transforms the sheaf of holomorphic sections $\an_Y$ into the sheaf of antiholomorphic sections $\ol{\an}_Y$. We denote the sheaf of ($\sigma$-)invariant holomorphic sections with $\an_Y^\sigma$: \em if $U\subset Y$ is open, then $H^0(U,\an_Y^\sigma):=\{F\in H^0(U,\an_Y):\ F=\ol{F}\circ\sigma\}$\em.

\subsubsection{Real part space}\label{fixed}
Let $(Y,\an_Y)$ be a complex analytic space endowed with an anti-involu\-tion $\sigma$. Let $Y^\sigma:=\{x\in Y:\ \sigma(x)=x\}$ and define the sheaf $\an_{Y^\sigma}:=\an^\sigma_Y|_{Y^\sigma}$. The $\R$-ringed space $(Y^\sigma,\an_{Y^\sigma})$ is called the \em real part space of $(Y,\an_Y)$ \em (with respect to $\sigma$), whereas $Y\setminus Y^\sigma$ is the \em imaginary part \em of $(Y,\an_Y)$. By \cite[Thm.II.4.10]{gmt} $(Y^\sigma,\an_{Y^\sigma})$ is a real analytic space if $Y^\sigma\neq\varnothing$.

\subsubsection{Complexification and $C$-analytic spaces \em \cite[\S III.3]{gmt}}\label{complexification} 
A real analytic space $(X,\an_X)$ is a \em $C$-analytic space \em if it satisfies one of the following two equivalent conditions:
\begin{itemize}
\item[(1)] Each local model of $(X,\an_X)$ is defined by a coherent sheaf of ideals, which is not necessarily associated to a \em well reduced structure \em (see \S\ref{cas}).
\item[(2)] There exists a complex analytic space $(Y,\an_{Y})$ endowed with an anti-involution $\sigma$ whose real part space is $(X,\an_X)$.
\end{itemize}
The complex analytic space $(Y,\an_{Y})$ is called a \em complexification \em of $X$ and it satisfies the following properties:
\begin{itemize}
\item[(i)] $\an_{Y,x}=\an_{X,x}\otimes\C$ for each $x\in X$.
\item[(ii)] The germ of $(Y,\an_{Y})$ at $X$ is unique up to an isomorphism.
\item[(iii)] $X$ has a fundamental system of invariant open Stein neighborhoods in $Y$. 
\item[(iv)] If $X$ is reduced, then $Y$ is also reduced.
\end{itemize} 
For further details see \cite{c,gmt,t,wb}. 

\subsubsection{Irreducible components of a $C$-analytic space}
A $C$-analytic space is \em irreducible \em if it is not the union of two $C$-analytic subspaces different from itself. In addition, $X$ is irreducible if and only if it admits a fundamental system of invariant irreducible complexifications. Given a $C$-analytic space $X$, there exists a unique irredundant (countable) locally finite family of irreducible $C$-analytic subspaces $\{X_i\}_{i\geq1}$ such that $X=\bigcup_{i\geq1}X_i$. The $C$-analytic subspaces $X_i$ are called the \em irreducible components of $X$\em. For further details see \cite{fe12,wb}. It is important to point out that $\mer(X)\cong\prod_{i\in I}\mer(X_i)$. Thus, when dealing with meromorphic functions, it is common to assume that $X$ is irreducible in order to study the properties of the ring $\mer(X)$. If such is the case, $\mer(X)$ is the quotient field of the integral domain $\an(X)$.

\subsubsection{Full approximation}
We now present the following result concerning approximation of continuous functions on the real part space of a Stein space (endowed with an anti-involution) by invariant holomorphic functions defined in the whole Stein space.

\begin{thm}[Full approximation]\label{fa}
Let $(Y,\an_Y)$ be a Stein space endowed with an anti-involu\-tion $\sigma$ and $(X,\an_X)$ its real part space. Let $g:X\to\R$ be a continuous function and $\veps:X\to\R$ a strictly positive continuous function. Then 
\begin{itemize}
\item[(i)] There exists an invariant holomorphic function $F:Y\to\C$ such that $|F|_{X}-g|<\veps$.
\item[(ii)] There exists an invariant holomorphic function $B:Y\to\C$ such that $0<B|_{X}<\veps$ and $\{B=0\}=\varnothing$.
\item[(iii)] If $\{y_j\}_{j\geq1}\subset Y$ is a discrete subset of $Y$, we may assume $F(y_j)\neq0$ for each $j\geq1$.
\end{itemize}
\end{thm}
\begin{proof}
(i) By \cite{n0,tto} there exist an integer $N\geq1$ and an injective proper invariant holomorphic map $\varphi:Y\to Z:=\varphi(Y)\subset\C^N$ such that $Z$ is a (global) analytic subset of $\C^N$ invariant under the usual conjugation $\ol{\,\cdot\,}$ of $\C^N$ and the following diagram is commutative
$$
\xymatrix{
Y\ar[r]^\varphi\ar@{<->}[d]_\sigma&Z\ar@{^{(}->}[r]\ar@{<->}[d]_{\ol{\,\cdot\,}}&\C^N\ar@{<->}[d]_{\ol{\,\cdot\,}}\\
Y\ar[r]^\varphi&Z\ar@{^{(}->}[r]&\C^N
}
$$
Let $Z^*:=\varphi(X)=\{z\in Z:\ z=\ol{z}\}$. Consider the continuous function $g':=g\circ(\varphi|_{X})^{-1}:Z^*\to\R$ and the strictly positive continuous function $\veps':=\veps\circ(\varphi|_{X})^{-1}:Z^*\to\R$. As $Z^*\subset\R^N$ is closed, there exist continuous extensions $g^*:\R^N\to\R$ of $g'$ and $\veps^*:\R^N\to\R$ of $\veps'$ to $\R^N$. Using a suitable partition of unity, we may assume $\veps^*$ is strictly positive on $\R^N$. By Whitney's approximation theorem \cite[\S1.6]{n4} there exists an analytic function $f^*:\R^N\to\R$ such that $|f^*-g^*|<\tfrac{1}{2}\veps^*$. This analytic function $f^*$ extends to an invariant (under conjugation) holomorphic function $F^*:\Omega\to\C$ where $\Omega$ is an invariant neighborhood of $\R^n$ in $\C^n$. Following the proof of Whitney's approximation theorem \cite[\S1.6]{n4}, one realizes that it is possible to assume $\Omega=\C^N$. 

To that end (see \cite[\S1.6.11, pag. 34-35]{n4}), it is enough to construct an exhaustion by compact sets $\{K_m\}_{m\geq1}$ of $\R^N$ (that is, $\R^N=\bigcup_{m\geq1}K_m$ and $K_m\subset\Int(K_{m+1})$ for each $m\geq1$) and invariant open neighborhoods $U_m\subset\C^N$ of $K_m$ such that $\C^N=\bigcup_{m\geq1}U_m$ and 
$$
{\rm Re}((z_1-x_1)^2+\cdots+(z_N-x_N)^2)>\tfrac{1}{2}\dist(K_m,\R^N\setminus K_{m+1})
$$
for each $z:=(z_1,\ldots,z_N)\in U_m$ and each $x:=(x_1,\ldots,x_N)\in\R^N\setminus K_{m+1}$. If $z_i:=a_i+\sqrt{-1}b_i$, $a:=(a_1,\ldots,a_N)$ and $b:=(b_1,\ldots,b_N)$, we obtain
$$
{\rm Re}((z_1-x_1)^2+\cdots+(z_N-x_N)^2)=\|a-x\|^2-\|b\|^2>\|a-x\|-\|b\|
$$
if $\|a-x\|+\|b\|>1$. Denote the closed (resp. open) ball of center $y\in\R^N$ and radio $r>0$ with $\ol{\Bb}_N(y,r)$ (resp. $\Bb_N(y,r)$). Define $K_m:=\ol{\Bb}_N(0,r_m)$ and
$$
U_m:=\{(a_1+\sqrt{-1}b_1,\ldots,a_N+\sqrt{-1}b_N):\ (a_1,\ldots,a_N)\in\Bb_N(0,r_m+\tfrac{1}{2}),\ b_i\in(-m,m)\ \forall\,i \}
$$
where $r_{m+1}:=r_m+1+2\sqrt{N}m$ and $r_0=0$ (that is, $r_m:=m+\sqrt{N}m(m+1)$). Observe that 
$$
\dist(K_m,\R^N\setminus K_{m+1})=r_{m+1}-r_m=1+2\sqrt{N}m
$$ 
and if $z\in U_m$ and $x\in\R^N\setminus K_{m+1}$, then
\begin{multline*}
{\rm Re}((z_1-x_1)^2+\cdots+(z_N-x_N)^2)>r_{m+1}-(r_m+\tfrac{1}{2})-\sqrt{N}m\\
=1+2\sqrt{N}m-\tfrac{1}{2}-\sqrt{N}m=\frac{1}{2}+\sqrt{N}m=\frac{1}{2}\dist(K_m,\R^N\setminus K_{m+1}).
\end{multline*}
For this choice the approximating analytic function $f^*:\R^n\to\R$ provided in \cite[\S1.6]{n4} extends to an invariant holomorphic function $F^*:\C^N\to\C$. Define $F:=F^*\circ\varphi:Y\to\C$, which is an invariant holomorphic function. We have
$$
|F|_{X}-g|=|F^*\circ\varphi|_{X}-g\circ(\varphi|_{X})^{-1}\circ\varphi|_{X}|=|f^*\circ\varphi|_{X}-g^*\circ\varphi|_{X}|<\tfrac{1}{2}\veps^*\circ\varphi|_{X}=\tfrac{1}{2}\veps,
$$
which proves (i).

For the rest of the proof we keep the notations introduced above.

(ii) Define 
$$
\veps_0^*:=\frac{\veps^*}{2(1+\veps^*)}<\frac{1}{2}.
$$
By (i) there is an invariant holomorphic function $A^*:\C^N\to\C$ such that $|A^*|_{\R^N}-\frac{1}{\veps_0^*}|<\veps_0^*$, so
$$
\frac{1}{2\veps_0^*}<\frac{1-(\veps_0^*)^2}{\veps_0^*}<A^*|_{\R^N}.
$$
Consequently, $B^*:=\exp(-1-(A^*)^2)$ satisfies
$$
0<B^*|_{\R^N}=\exp(-1-(A^*)^2)|_{\R^N}<\frac{1}{1+(A^*)^2|_{\R^N}}\leq\frac{1}{2A^*|_{\R^N}}<\veps^*_0<\veps^*
$$
and $\{B^*=0\}=\varnothing$.

(iii) Let us modify the $F^*$ constructed in (i). If $F^*=0$, then $|g^*|<\tfrac{1}{2}\veps^*$. By (ii) there exists an invariant holomorphic function $B^*:\C^N\to\C$ such that $0<B^*|_{\R^N}<\tfrac{1}{2}\veps^*$ and $\{B^*=0\}=\varnothing$. Thus, $|g^*-B^*|_{\R^N}|\leq|g^*|+|B^*|_{\R^N}|<\veps^*$ and $B^*(y_j)\neq0$ for each $j\geq1$. Assume in the following $F^*\neq0$. By \cite[\S2.5. Ex.10, pp. 64-65]{hir} there exists a strictly positive continuous function $\delta^*:\R^N\to\R$ such that if $\varphi:=(\varphi_1,\ldots,\varphi_N):\R^N\to\R^n$ is a continuous map and satisfies $|\varphi_i(x)-x_i|<\delta^*(x)$ for each $x\in\R^n$ and each $i=1,\ldots,N$, then $|F^*|_{\R^N}-F^*|_{\R^N}\circ\varphi|<\tfrac{1}{2}\veps^*$. By (ii) there exists an invariant holomorphic function $B^*:\C^N\to\C$ such that $0<B^*|_{\R^N}<\delta^*$ and $\{B^*=0\}=\varnothing$. For each $\lambda:=(\lambda_1,\ldots,\lambda_N)\in\C^N$ define $\Phi_\lambda:=\id_\C^N+\lambda B^*$. If $\lambda\in(-1,1)^N$, we have $|F^*|_{\R^N}-F^*|_{\R^N}\circ\Phi_\lambda|<\tfrac{1}{2}\veps^*$.
Fix $j\geq1$. As $B^*(y_j)\neq0$ and $\{F^*=0\}$ is a hypersurface of $\C^N$, also $\{F^*(y_j+\lambda B^*(y_j))=0\}$ is a hypersurface of $\C^N$. Thus, $\Ss_j:=\{F^*(y_j+\lambda B^*(y_j))=0\}\cap\R^N$ is a $C$-analytic set of dimension $\leq N-1$. As this happens for each $j\geq1$, we have $\Tt:=(-1,1)^N\setminus\bigcup_{j\geq1}\Ss_j\neq\varnothing$ and it is enough to pick $\lambda\in\Tt$ to have in addition $F^*\circ\Phi_\lambda(y_j)\neq0$ for each $j\geq1$, as required.
\end{proof}

\subsection{$C$-analytic sets}\label{cas} 
The concept of $C$-analytic sets was introduced by Cartan in \cite[\S7,\S10]{c}. Recall that a subset $X\subset\R^n$ is \em $C$-analytic \em if there exists a finite set $\Ss:=\{f_1,\ldots,f_r\}$ of real analytic functions $f_i$ on $\R^n$ such that $X$ is the common zero-set of $\Ss$. This property is equivalent to the following:
\begin{itemize}
\item[(1)] There exists a coherent sheaf of ideals $\Ii$ on $\R^n$ such that $X$ is the zero-set of $\Ii$.
\item[(2)] There exist an open neighborhood $\Omega$ of $\R^n$ in $\C^n$ and a complex analytic subset $Z$ of $\Omega$ such that $Z\cap\R^n=X$.
\end{itemize}

\subsubsection{Well reduced structure}
Given a $C$-analytic set $X\subset\R^n$, the largest coherent sheaf of ideals $\Ii$ having $X$ as zero-set is $\Jj(X)\an_{\R^n}$ by Cartan's Theorem $A$, where $\Jj(X)$ is the set of all analytic functions on $\R^n$ that are identically zero on $X$. The coherent sheaf $\an_X:=\an_{\R^n}/\Jj(X)\an_{\R^n}$ is called the \em well reduced structure of $X$\em. The $C$-analytic set $X$ endowed with its well reduced structure is a $C$-analytic space, so it has a well-defined complexification as commented above. By \cite{n0,tto} the only limitation that a $C$-analytic space has to be embedded in some $\R^N$ as a $C$-analytic set is the local embedding dimension, which has to be bounded by a common constant at all points of the $C$-analytic space.

\subsubsection{Set of points of non-coherence of a $C$-analytic set}
A $C$-analytic set $X\subset\R^n$ is {\em coherent} if the sheaf $\Jj_X$ of germs of analytic functions vanishing identically on $X$ is an $\an_{\R^n}$-coherent sheaf of modules. As $\an_{\R^n}$ is a coherent sheaf of rings, $\Jj_X$ is \em coherent \em at $x\in\R^n$ if and only if it is of \em finite type \em at $x\in\R^n$, that is, there exists an open neighborhood $U$ of $x$ in $\R^n$ and finitely many sections $f_i\in H^0(U,\Jj_X)$ such that for each $y\in U$ the germs $f_{1,y},\ldots,f_{r,y}$ generate the stalk $\Jj_{X,y}$ as an $\an_{X,y}$-module. By Cartan's Theorem A $X$ is coherent at $x\in X$ if and only if $\Jj_{X,x}=\Jj(X)\an_{\R^n,x}$. Thus, $X$ is {\em non-coherent at $x\in X$} if and only if $\Jj_{X,x}\neq\Jj(X)\an_{\R^n,x}$. We refer the reader to \cite{abf4,fe,gal,tt} for a precise study of the set $N(X)$ of points of non-coherence of the $C$-analytic set $X$. Classical examples of non-coherent $C$-analytic sets are Cartan's or Whitney's umbrellas.

\subsection{Singular and regular points}
Let $(Y,\an_Y)$ be a reduced complex analytic space and $y\in Y$. We say that $y\in Y$ is a \em regular point of $Y$ \em if $\an_{Y,y}$ is a regular local ring. Otherwise, we say $y\in Y$ is a \em singular point of $Y$\em. In the (reduced) complex case the regularity of $\an_{Y,y}$ is equivalent to the smoothness of the analytic set germ $Y_y$ (use for instance the Jacobian criterion). We denote the set of regular points of $Y$ with $\Reg(Y)$ and the set of singular points of $Y$ with $\Sing(Y)$.

\subsubsection{Singular set of a $C$-analytic space}\label{sing}

Let $(X,\an_X)$ be a $C$-analytic space and $(Y,\an_Y)$ a complexification of $X$. We define the \em singular locus of $X$ \em as $\Sing(X):=\Sing(Y)\cap X$. Observe that $(\Sing(X),\an_X|_{\Sing(X)})$ is a $C$-analytic space of dimension strictly smaller than the dimension of $X$. The complement $\Reg(X):=X\setminus\Sing(X)$ is the set of \em regular points of $X$ \em and corresponds to the set of those points $x\in X$ such that the stalk $\an_{X,x}$ is a local regular ring (recall here that $\an_{X,x}$ is a regular local ring if and only if its complexification $\an_{X,x}\otimes_\R\C$ is a regular local ring). Thus, $\Sing(X)$ depends only on $X$ and not in the chosen complexification.

We say that $x\in X$ is a \em smooth point \em of $X$ if there exists a neighborhood $U\subset X$ of $x$ such that $U$ is analytically diffeomorphic to an open subset of $\R^n$. It is important to distinguish between regular and smooth points because in the real case they are different concepts (\cite[Ex.2.2]{abf4} or \cite[Ex.2.1]{fe14}), although each regular point is always a smooth point. 

\subsection{Normalization of complex analytic spaces}\label{norm}
One defines the normalization of a complex analytic space in the following way \cite[VI.2]{n1} (see also \cite{abf5,hp} for an equivalent approach for Stein spaces that involves integral closure and Fr\'echet's topology). A complex analytic space $(Y,\an_Y)$ is \em normal \em if for each $y\in Y$ the local analytic ring $\an_{Y,y}$ is reduced and integrally closed. A \em normalization \em $((\widehat{Y},\an_{\widehat{Y}}),\pi)$ of a complex analytic space $(Y,\an_Y)$ is a normal complex analytic space $(\widehat{Y},\an_{\widehat{Y}})$ endowed with a proper surjective holomorphic map $\pi:\widehat{Y}\to Y$ with finite fibers such that $\widehat{Y}\setminus\pi^{-1}(\Sing(Y))$ is dense in $\widehat{Y}$ and the restriction $\pi|:\widehat{Y}\setminus\pi^{-1}(\Sing(Y))\to Y\setminus\Sing(Y)$ is an analytic isomorphism. The normalization $((\widehat{Y},\an_{\widehat{Y}}),\pi)$ of a reduced complex analytic space $(Y,\an_Y)$ always exists and is unique up to isomorphism \cite[VI.2.Lem.2 \& VI.3.Thm.4]{n1}. 

\subsubsection{Complex analytic spaces endowed with an anti-involution}\label{casai}
Let $(Y,\an_Y)$ be a (reduced) complex analytic space endowed with an anti-involution $\sigma$ and $((\widehat{Y},\an_{\widehat{Y}}),\pi)$ its normalization. By \cite[Prop.IV.3.10]{gmt} there exists an anti-involution $\widehat{\sigma}$ of $\widehat{Y}$ that makes the following diagram commutative
$$
\xymatrix{
\widehat{Y}\ar[r]^{\widehat{\sigma}}\ar[d]_\pi&\widehat{Y}\ar[d]^\pi\\
Y\ar[r]^\sigma&Y}
$$
Denote the real part space of $Y$ with $X:=Y^\sigma$ and the real part space of $\widehat{Y}$ with $\widehat{X}:=\widehat{Y}^{\widehat{\sigma}}$. Let $\an_Y^\sigma$ and $\an_{\widehat{Y}}^{\widehat{\sigma}}$ be the subsheaves of invariant holomorphic functions of $\an_Y$ and $\an_{\widehat{Y}}$. The pairs $(X,\an_X:=\an_Y^\sigma|_X)$ and $(\widehat{X},\an_{\widehat{X}}:=\an_{\widehat{Y}}^{\widehat{\sigma}}|_{\widehat{X}})$ are $C$-analytic spaces and $(Y,\an_Y)$ and $(\widehat{Y},\an_{\widehat{Y}})$ are corresponding complexifications for these real analytic spaces. The restriction map $\pi|_{\widehat{X}}:\widehat{X}\to X$ induces a real analytic morphism between $(\widehat{X},\an_{\widehat{X}})$ and $(X,\an_X)$. The restriction map
$$
\pi|_{\widehat{X}\setminus\pi^{-1}(\Sing(X))}:\widehat{X}\setminus\pi^{-1}(\Sing(X))\to X\setminus\Sing(X)
$$
is an analytic diffeomorphism and, as $\pi:\widehat{X}\to X$ is proper, $
\cl(X\setminus\Sing(X))\subset\pi(\widehat{X})\subset X$. Observe that $\cl(X\setminus\Sing(X))$ is the set of points $x\in X$ such that $\dim(X_x)=\dim(X)$. Thus, if $X$ is pure dimensional, $X=\cl(X\setminus\Sing(X))$ and $\pi:\widehat{X}\to X$ is surjective.

In addition, $\widehat{X}\subset\pi^{-1}(X)$ and the previous inclusion in general could be strict. By \cite[Lem.IV.3.13]{gmt} the equality $\widehat{X}=\pi^{-1}(X)$ is guaranteed if $X$ is a coherent $C$-analytic space. In addition, if such is the case, $(\widehat{X},\an_{\widehat{X}})$ is by \cite[Thm.IV.3.14]{gmt} a coherent real analytic space and $(\widehat{Y},\an_{\widehat{Y}})$ is a (normal) complexification of $(\widehat{X},\an_{\widehat{X}})$.

\subsubsection{$C$-analytic spaces of dimension $\leq2$}\label{casai2}
If $(X,\an_X)$ is a $C$-analytic space, the normality of a complexification $(Y,\an_Y)$ and the coherence of $(X,\an_X)$ are independent concepts if $X$ has dimension $\geq3$ (\cite[Esempio, p. 211]{t}). Let us analyze what happens if the dimension of $X$ is $\leq2$. If $X$ has dimension $1$, then it is coherent \cite{fe,gal}. In addition, the normality of a complexification $(Y,\an_Y)$ implies that $Y$ is non-singular \cite[Thm.VI.2.2]{n1}, so $X$ is also non-singular. 

If $X$ has dimension $2$, it is not coherent in general, but the normality of a complexification $(Y,\an_Y)$ implies that $Y$ has only isolated singularities and by \cite[Prop.III.2.8]{gmt} or \cite[V.\S1.Prop.5, pag.94]{n1} we deduce $(X,\an_X)$ is coherent.

In \cite[\S5]{abf4} we provide a careful study of the set $N(X)$ of points of non-coherence of a $C$-analytic set $X\subset\R^n$. We keep the notations introduced in \S\ref{casai}. If $X$ is a $C$-analytic space of dimension $2$, then $N(X)=\pi(\cl(\pi^{-1}(X)\setminus\widehat{X})\cap\widehat{X})$, which is the discrete subset of $X$. 

\begin{remarks}\label{casd2}
Let $(X,\an_X)$ be a $C$-analytic space of dimension $2$. Let us keep the notations introduced in \S\ref{casai} and \S\ref{casai2}.

(i) \em Each connected component of $\cl(\pi^{-1}(X)\setminus\widehat{X})$ meets $\widehat{X}$\em.

Suppose $\cl(\pi^{-1}(X)\setminus\widehat{X})$ has a connected component $C_1$ that does not meet $\widehat{X}$. Let $S:=\pi^{-1}(X)\setminus(C_1\cup\widehat{\sigma}(C_1))$ and $W\subset\widehat{Y}$ an invariant neighborhood of $S$ such that $\cl(W)\cap C_1=\varnothing$. Let $U$ be an open neighborhood of $C_1$ that does not meet $\cl(W)$. Define $T:=\widehat{Y}\setminus(W\cup U\cup\widehat{\sigma}(U))$, which is a closed subset of $\widehat{Y}$ that does not meet $\pi^{-1}(X)$. As $\pi$ is proper, $\Omega:=\pi^{-1}(Y\setminus\pi(T))$ is an open neighborhood of $\pi^{-1}(X)$. Substitute $\widehat{Y}$ by $\Omega$ and $Y$ by $\pi(\Omega)=Y\setminus\pi(T)$. Let $U'$ be the connected component of $\widehat{Y}$ that contains $C_1$ and does not meet $\widehat{X}$. Note that $U'$ has dimension $\geq1$ because $(Y,\an_Y)$ is a reduced complex analytic space. As $C_1$ does not meet $\widehat{X}$, we deduce $U'\cap\widehat{\sigma}(U')=\varnothing$. Then $Y_1:=\pi(U')$ and $Y_2:=\pi(\widehat{\sigma}(U'))$ are irreducible components of $Y$ and $Y_1\neq Y_2$. Observe that $Y_1\cap X=\pi(C_1)=\pi(\widehat{\sigma}(C_1))=Y_2\cap X$ (recall that $\pi\circ\widehat{\sigma}=\sigma\circ\pi$). Define $Y':=\pi(Y\setminus(U'\cup\widehat{\sigma}(U')))\cup(Y_1\cap Y_2)\subset Y$ and observe that $X$ is the real part space of the reduced complex analytic space $(Y',\an_{Y'})$ endowed with the anti-involution $\sigma|_{Y'}:Y'\to Y'$. If $x\in\pi(C_1)$, then $\an_{Y',x}=\an_{X,x}\otimes_\R\C=\an_{Y,x}$, so $Y'_x=Y_x$, which is a contradiction because $Y_1\neq Y_2$. Thus, each connected component of $\cl(\pi^{-1}(X)\setminus\widehat{X})$ meets $\widehat{X}$, as required.

(ii) \em $X$ is coherent if and only if $\pi^{-1}(X)\setminus\widehat{X}=\varnothing$\em. 

If $\pi^{-1}(X)\setminus\widehat{X}=\varnothing$, then $N(X)=\pi(\cl(\pi^{-1}(X)\setminus\widehat{X})\cap\widehat{X})=\varnothing$. Suppose $X$ is coherent and $\pi^{-1}(X)\setminus\widehat{X}\neq\varnothing$. By (i) each connected component of $\cl(\pi^{-1}(X)\setminus\widehat{X})$ meets $\widehat{X}$, so $N(X)=\pi(\cl(\pi^{-1}(X)\setminus\widehat{X})\cap\widehat{X})\neq\varnothing$, which is a contradiction. Thus, $\pi^{-1}(X)\setminus\widehat{X}=\varnothing$, as required.
\qed
\end{remarks}

\begin{cor}\label{dim12}
If $X$ is a non-coherent irreducible $C$-analytic space of dimension $2$, then $\widehat{X}$ is pure dimensional, $C:=\cl(\pi^{-1}(X)\setminus\widehat{X})$ is an invariant analytic set of real dimension $1$ such that $\pi^{-1}(X)=\widehat{X}\cup C$, $\widehat{X}\cap C$ is a discrete subset of $\widehat{X}$ and the irreducible components of $C$ have dimension $1$. In addition, $\pi(\widehat{X})=\cl(X\setminus\Sing(X))$ and $\cl(X\setminus\pi(\widehat{X}))\subset\pi(C)$.
\end{cor}
\begin{proof}
As $(\widehat{Y},\an_{\widehat{Y}})$ is a normal complex analytic space of dimension $2$, the analytic set $\Sing(\widehat{Y})$ has dimension $0$, so $\Sing(\widehat{X})=\Sing(\widehat{Y})\cap\widehat{X}$ has also dimension $0$ and $\widehat{X}$ is pure dimensional. 

As $\pi$ has finite fibers and $\pi\circ\widehat{\sigma}=\sigma\circ\pi$, $\pi^{-1}(\Sing(X))$ is an analytic set of real dimension $\leq1$ and it is invariant. Let $C$ be the union of the irreducible components of $\pi^{-1}(\Sing(X))$ (as an analytic set of real dimension $1$) that are not contained in $\widehat{X}$. Then $C$ is invariant, $\pi^{-1}(X)=\widehat{X}\cup C$ and $\widehat{X}\cap C$ is a discrete set. As $\pi^{-1}(X\setminus\Sing(X))=\widehat{X}\setminus\pi^{-1}(\Sing(X))$,
$$
C=\cl(\pi^{-1}(\Sing(X))\setminus\widehat{X})=\cl(\pi^{-1}(X)\setminus\widehat{X}).
$$
As $X$ is non-coherent, $C\neq\varnothing$ (Remark \ref{casd2}(ii)). By Remark \ref{casd2}(i) each connected component of $C$ meets $\widehat{X}$, so $C$ has no isolated points. Thus, the irreducible components of $C$ have dimension $1$. As $\widehat{X}\setminus \pi^{-1}(\Sing(X))$ is dense in $\widehat{X}$ and $\pi:\widehat{X}\to X$ is proper,
$$
\pi(\widehat{X})=\cl(\pi(\widehat{X}\setminus \pi^{-1}(\Sing(X))))=\cl(X\setminus\Sing(X)).
$$
As $\pi^{-1}(X)=\widehat{X}\cup C$, we conclude $\cl(X\setminus\pi(\widehat{X}))\subset\pi(C)$, as required.
\end{proof}

\subsection{Multiplicity along an irreducible complex analytic curve}

Let $(Y,\an_Y)$ be a normal Stein space of pure dimension $2$. Recall that $\Sing(Y)$ is a discrete set. Let $Z\subset Y$ be an irreducible analytic subspace of dimension $1$ (recall that $\an_Z:=\an_Y/\Jj_Z$). Let us define the \em multiplicity ${\rm mult}_Z(F)$ along $Z$ \em of $F\in\an(Y)$ (see \cite{ad} and \cite[\S2]{adr} for similar results in the real analytic setting).

Fix a point $z\in Z\setminus\Sing(Y)$. As $\an_{Z,z}$ is a unique factorization domain, the ideal $\Jj_{Z,z}$ is principal, say $\Jj_{Z,z}=H_z\an_{Y,z}$ and $F_z=A_zH_z^{\alpha_z}$ for some $A_z\in\an_{Y,z}\setminus \Jj_{Z,z}$ and some integer $\alpha_Z\ge0$. This integer $\alpha_z$ is the \em multiplicity ${\rm mult}_Z(F)$\em. Let us show that this definition is consistent: 

\subsubsection{}\label{paso2} \em The number $\alpha_z$ does not depend on the chosen generator of $\Jj_{Z,z}$\em. 

Assume $F_z=B_z G_z^{\beta_z}$ for another generator $G_z$ of $\Jj_{Z,z}$ and $B_z\in\an_{Y,z}\setminus\Jj_{Z,z}$. Suppose 
$\alpha_z\le\beta_z$ and write $H_z=C_zG_z$ for some $C_z\in\an_{Y,z}\setminus\Jj_{Z,z}$. As $B_z G_z^{\beta_z}=A_z(C_zG_z)^{\alpha_z}$, we have $B_z G_z^{\beta_z-\alpha_z}=A_zC_z^{\alpha_z}\notin\Jj_{Y,z}$, so $\beta_z-\alpha_z=0$. 

\subsubsection{}\label{paso3} \em The number $\alpha_z$ is constant on some branch of the analytic curve germ $Z_z$\em. 

If $y\in Z$ is close to $z$, we have $F_y=A_yH_y^{\alpha_z}$ and $H_y$ is a generator of $\Jj_{Z,y}$ (recall that Stein spaces are coherent). As $A_z\notin\Jj_{Z,z}$, $A$ cannot vanish identically on all branches of $Z_z$, so that $A(y)\ne0$ for $y$ close to $z$ in the branches of $Z_z$ on which $A_z$ is not identically zero. We conclude $\alpha_y=\alpha_z$ for each $y\in Z$ close to $z$ such that $A(y)\neq0$.

\subsubsection{}\label{paso4} \em There exists a global analytic function $G\in\an(Y)$ that vanishes on $Z$ and whose germ at $z$ generates $\Jj_{Z,z}$ for each $z\in Z$ outside a discrete set $D\subset Z$\em. We may assume in addition that $\{G=0\}$ avoids any discrete subset $E$ of $Y$ that does not meet $Z$.

Pick $a\in Z\setminus(\Sing(Y)\cup\Sing(Z))$. By Cartan's Theorem A the ideal $\Jj_{Z,a}$ is generated by finitely many $F_i\in\an(Y)$ that vanish identically on $Z$. We may assume $F_{1,a}\in\Jj_{Z,a}\setminus\Jj_{Z,a}^2$. Define $G:=F_1$ and observe that $G_a=A_a H_a$ where $A_a\notin\Jj_{Z,a}$. By \S\ref{paso3} $a$ is isolated in $Z\cap\{A=0\}$, hence $G_y$ generates $\Jj_{Z,y}$ for each $y\in Z\setminus\{a\}$ close enough to $a$. Consider the coherent sheaf of ideals $\Ii:=(G\an_{Y}:\Jj_Z)$ of $\an_Y$. Observe that $\Ii_y=\an_{Y,y}$ if and only if $G_y$ generates $\Jj_{Z,y}$. Thus, the support 
$$
Z':=\supp(\an_Y/\Ii)=\{y\in Y:\,G_y \text{ does not generate } \Jj_{Z,y}\}
$$
is a closed analytic subspace of $Y$ that does not contain $a$. As $Z$ is irreducible, $D:=Z'\cap Z$ is a discrete set, so $G_z$ generates $\Jj_{Z,z}$ for each $z\in Z\setminus D$.

Let $E_1:=\{z\in E:\ G(z)=0\}$ and $E_2:=\{z\in E:\ G(z)\neq0\}$. By Cartan's Theorem B there exists an invariant holomorphic function $H\in\an(Y)$ such that $H|_Z=0$, $H(z)=1$ for each $z\in E_1$ and $H(z)=0$ for each $z\in E_2$. Then $G+H^2$ satisfies the required properties.

\subsubsection{}\label{paso5} 
Fix $x\in Z\setminus(\Sing(Y)\cup\Sing(Z)\cup D)$. We have $F_x=A_x G_x^{\alpha_x}$ with $A_x\notin\Jj_{Z,x}$. By \S\ref{paso3} $A_x$ is a unit for $z\in Z\setminus D$ close to $x$. Consequently, $\alpha_z=\alpha_x$, $G_z^{\alpha_x}|F_z$ and $F_z|G_z^{\alpha_x}$. Consider the coherent sheaf of ideals $\Jj:=(G^{\alpha_x}:F)\cap(F:G^{\alpha_x})$ of $\an_Y$. The support 
$$
Z'':=\supp(\an_Y/\Jj)=\{z\in Y:\, G_z^{\alpha_x}\!\not|\,F_z\text{ or } F_z\!\not|\,G_z^{\alpha_x}\}
$$
is an analytic subset of $Y$ that does not contain $x$, so $D':=Z\cap Z''$ is a discrete set because $Z$ is irreducible. Thus, $\alpha_z=\alpha_x$ for each $z\in Y\setminus(\Sing(Y)\cup D\cup D')$. 

\subsubsection{}\label{paso5b} 
If $x\in Z\setminus\Sing(Y)$, we conclude by \S\ref{paso3} and \S\ref{paso4} that $\alpha_x=\alpha_z$ for each $z$ close to $x$. Hence, by \S\ref{paso4} $\alpha_x=\alpha_z$ if $x,z\in Z\setminus\Sing(Y)$.

\subsubsection{}\label{paso6} Consequently, ${\rm mult}_Z(F)$ is well-defined and: \em If ${\rm mult}_Z(F)=\alpha$, there exists a discrete set $E$ such that $F_z/ G_z^{\alpha}$ is a unit for $z\in Z\setminus E$\em. If $F$ does not vanish identically at $Z$, we write ${\rm mult}_Z(F)=0$.

\subsubsection{}\label{paso7} 
\em It is possible to modify the analytic function $G$ constructed in \em \ref{paso4} \em to avoid any prescribed analytic subspace $Z'\subset Y$ of dimension $1$ that meets $Z$ in a discrete set\em. 

Assume we have constructed $H,F\in\an(Y)$ such that
\begin{itemize}
\item[(i)] $Z\subset\{H=0\}$ and $\{H=0\}\cap Z'$ is a discrete set.
\item[(ii)] $Z'\subset\{F=0\}$ and $\{F=0\}\cap Z$ is a discrete set.
\end{itemize}
Define $G':=FG+H^2$ and observe that $Z\subset\{G'=0\}$ and $Z'\cap\{G'=0\}=Z'\cap\{H=0\}$ is a discrete set. As ${\rm mult}_Z(H^2)\geq2$, we deduce ${\rm mult}_Z(G')=1$.

To finish we show how to construct $H$. The construction of $F$ is analogous. As $Z$ is a coherent analytic subspace, there exist by \cite{fr} finitely many $H_i\in\an(Y)$ such that $Z=\{H_1=0,\ldots,H_s=0\}$. For each irreducible component $Z_j'$ of $Z'$ of dimension $1$ we choose one of its points $z_j\not\in Z$. The set $\{z_j\}_{j\geq1}$ constitutes a discrete subset of $Z'$. For each $j$ consider the linear equation
$$
\x_1H_1(z_j)+\cdots+\x_sH_s(z_j)=0,
$$
which is not identically zero because $z_j\not\in Z=\{H_1=0,\ldots,H_s=0\}$. Thus, $L_j:=\{\x_1H_1(z_j)+\cdots+\x_sH_s(z_j)=0\}$ is a hyperplane for each $j$ and $\C^\s\setminus\bigcup_{j\geq1}L_j\neq\varnothing$. Let $\lambda:=(\lambda_1,\cdots,\lambda_s)\in\C^\s\setminus\bigcup_{j\geq1}L_j$. Then, $H:=\lambda_1H_1+\cdots+\lambda_sH_s$ satisfies $Z\subset\{H=0\}$ and $\{H=0\}\cap Z'=\varnothing$.

\section{Proof of Theorem \ref{h17}}\label{s3}

Let $(Y,\an_Y)$ be a reduced complex analytic space endowed with an anti-involution $\sigma:Y\to Y$. Consider the $C$-analytic space $(X:=Y^\sigma,\an_X)$. Let $C\subset Y$ be an invariant analytic space of real dimension $1$ without isolated points. Assume $C$ is a closed subset of $Y$ and $D:=X\cap C$ is a discrete set. We write $C=\bigcup_{i\geq1} C_i$, where the $C_i$ are the irreducible components of $C$ (as an analytic space of real dimension $1$). Recall that $C_i$ is isomorphic either to $\R$ or $\sph^1$ via an analytic parameterization (use normalization). In the second case we compose with $\varphi:\R\to\sph^1,\ t\mapsto(\cos(t),\sin(t))$. Thus, for each $i\geq1$ there exists a surjective analytic map $\alpha_i:\R\to C_i$. 

\subsection{Eliminating one dimensional zero-sets from the real part space}
We assume in the following that $X\cup C$ has a system of open invariant Stein neighborhoods in $Y$. A paracompact locally connected and locally compact space $Y$ with countably many connected components (which are clopen subsets of $X$) has an exhaustion by compact sets, that is, there exists a family of compact sets $\{K_m\}_{m\geq1}$ such that $X=\bigcup_{m\geq1}K_m$ and $K_m\subset\Int(K_{m+1})$ for each $m\geq1$. Thus, reduced complex analytic spaces admit exhaustions by compact sets.

\begin{lem}\label{dim1}
Assume $(Y,\an_Y)$ is a normal Stein space of dimension $2$ and let $F\in\an(Y)$ be an invariant holomorphic function such that $F|_X\geq0$ and $\{F|_X=0\}$ has dimension $\leq1$. Then there exist invariant holomorphic functions $G,H_1,H_2\in\an(Y)$ such that $\{H_i|_X=0\}$ is a discrete set contained in $\{F|_X=0\}$, $H_i|_X\geq0$ and $H_1^2F=(G^2+F^2)H_2$.
\end{lem}
\begin{proof}
Let $Z\subset Y$ be the complexification of the $1$-dimensional part of $\{F|_X=0\}$ and $\{Z_j\}_{j\geq1}$ the collection of the irreducible components of $Z$ (as a complex analytic space). For each $j\geq1$ the intersection $Z_j\cap X$ has dimension $1$ (so $Z_j$ is in particular invariant). Thus, ${\rm mult}_{Z_j}(F)=2m_j$ for some $m_j\geq1$. As the family $\{Z_j\}_{j\geq1}$ is locally finite, $\Jj_{Z_j,z}=\an_{Y,z}$ for each $z\in Y$ and each $j\geq1$ except for at most finitely many indices. Consider the coherent sheaf of ideals $\Ff:=\prod_{j\geq1}\Jj_{Z_j}^{m_j}$ of $\an_{Y}$. Let $\{K_\ell\}_{\ell\geq1}$ be an exhaustion of $Y$ by compact sets. We may assume $Z_\ell\cap\Int(K_\ell)\neq\varnothing$ and pick a point 
$$
x_\ell\in Z_\ell\setminus\Big(\bigcup_{k\neq\ell}Z_k\cup\Sing(Z_\ell)\cup\Sing(Y)\Big). 
$$
Using Cartan's Theorem A and the previous exhaustion we find a countable system $\{G_\ell\}_{\ell\geq1}\subset\an(Y)$ of global generators of $\Ff$. We reorder the indices $\ell$ in such a way that ${\rm mult}_{Z_\ell}(G_\ell)=m_\ell$ and $G_{\ell,x_\ell}\in\Jj_{Z_\ell,\x_\ell}^{m_\ell}\setminus\Jj_{Z_\ell,\x_\ell}^{m_\ell+1}$. Observe that $G_{k,x_\ell}\in\Jj_{Z_\ell,\x_\ell}^{m_\ell}$ for each $k,\ell\geq1$. Define $\mu_\ell:=\max_{K_\ell}\{|G_\ell|\}+1$ and substitute $G_\ell$ by $\frac{1}{2^\ell\mu_\ell}G_\ell$ in order to have $\max_{K_\ell}\{|G_\ell|\}<\frac{1}{2^\ell}$. Let us show how to pick suitable values $0<\lambda_\ell<1$ such that: \em $G:=\sum_{\ell\geq1}\lambda_\ell G_\ell$ satisfy ${\rm mult}_{Z_k}(G)=m_k$ for each $k\geq1$ and $\{F=0\}\cap\{G=0\}\cap X\subset Z$\em. Thus, ${\rm mult}_{Z_\ell}(G^2+F^2)=2m_\ell$ for each $\ell\geq1$.

Let us choose $(\lambda_\ell)_{\ell\geq1}\in\Lambda:=\prod_{\ell\geq1}(0,1)$ to guarantee $G_{x_k}\in\Jj_{Z_k,\x_k}^{m_k}\setminus\Jj_{Z_k,\x_k}^{m_k+1}$ for each $k\geq1$. Write $G_{\ell,x_k}=G_{k,x_k}^{m_k}\zeta_{\ell,x_k}$, where $\zeta_{\ell,x_k}\in\an_{Y,\x_k}$ is an invariant holomorphic germ for each $\ell\geq1$ and $\zeta_{k,x_k}=1$. Observe that $G_{x_k}=G_{k,x_k}^{m_k}(\sum_{\ell\geq1}\lambda_\ell\zeta_{\ell,x_k})$. Consequently, we need $\sum_{\ell\geq1}\lambda_\ell\zeta_{\ell,x_\ell}(x_k)\neq0$ for each $k\geq1$. Define $\Ss_k:=\{(\lambda_\ell)_{\ell\geq1}\in\Lambda:\ \sum_{\ell\geq1}\lambda_\ell\zeta_{\ell,x_\ell}(x_k)=0\}$.

As $D:=\{F=0\}\cap(X\setminus Z)=\{y_m\}_{m\geq1}$ is a discrete set, to have $\{F=0\}\cap\{G=0\}\cap X\subset Z$ it is enough to choose $(\lambda_\ell)_{\ell\geq1}\in\Lambda$ in such a way that $G(y_m)=\sum_{\ell\geq1}\lambda_\ell G_\ell(y_m)\neq0$ for each $m\geq1$. Define $\Tt_m:=\{(\lambda_\ell)_{\ell\geq1}\in\Lambda:\ \sum_{\ell\geq1}\lambda_\ell G_\ell(y_m)=0\}$.

Consequently, it is enough to pick $(\lambda_\ell)_{\ell\geq1}\in\Lambda\setminus(\bigcup_{k\geq1}\Ss_k\cup\bigcup_{m\geq1}\Tt_m)$.

Consider the coherent sheaf of ideals $\Gg:=(F:G^2+F^2)\an_Y$ whose stalks are
$$
\Gg_z=\{\xi_z\in\an_{Y_z}:\ \xi_z(G^2_z+F^2_z)\in F_z\an_{Y_z}\}.
$$
We have $\Gg_z=\an_{Y_z}$ if and only if $F_z$ divides $G^2_z+F^2_z$. The support of the coherent sheaf $\an_Y/\Gg$ meets $X$ in the discrete set of points $z\in X$ at which $F_z$ does not divide $G^2_z+F^2_z$. 

Using Cartan's Theorem A and the exhaustion $\{K_\ell\}_{\ell\geq1}$ we find a countable system $\{A_\ell\}_{\ell\geq1}\subset\an(Y)$ of (global) invariant generators of $\Gg$. Define $\rho_\ell:=\max_{K_\ell}\{|A_\ell|\}+1$ and substitute $A_\ell$ by $\frac{1}{2^\ell\rho_\ell}A_\ell$ in order to have $\max_{K_\ell}\{|A_\ell|\}<\frac{1}{2^\ell}$. Define $A:=\sum_{\ell\geq1}A_\ell^2$, which is an invariant holomorphic function such that $A|_X\geq0$ and it has a discrete zero-set contained in $\{F|_X=0\}$. Observe that $A':=A(G^2+F^2)/F\in\an(Y)$ is an invariant holomorphic function such that $A'|_X\geq0$ and it has a discrete zero-set contained in $\{F|_X=0\}$. Thus, $A'^2F=A'A(G^2+F^2)$ and it is enough to take $H_1:=A'$ and $H_2:=A'A$.
\end{proof}

\subsection{Eliminating one dimensional zero-sets from the imaginary part}
In view of Lemma \ref{dim1} it is enough to represent invariant holomorphic functions $F\in\an(Y)$ as sums of squares of meromorphic functions such that $F|_X\geq0$ and $\{F|_X=0\}$ is a discrete set. Next, we prove that that we may assume in addition that $F|_C$ is also a discrete set. Before we state such result we propose the following preliminary lemma.

\begin{lem}\label{red0}
Assume $(Y,\an_Y)$ is a Stein space and let $F,G\in\an(Y)$ be invariant holomorphic functions. Let $\{y_m\}_{m\geq1}$ be a discrete set such that $\ol{G}(y_m)F(y_m)\neq0$ for each $m\geq1$. Let $\{a_m\}_{m\geq1}\subset\C$ be such that $a_m\in\R$ if $\sigma(y_m)=y_m$. Then there exists an invariant holomorphic function $A\in\an(Y)$ with empty zero-set such that $(\ol{G}FA^2)(y_m)=a_m$ for each $m\geq1$.
\end{lem}
\begin{proof}
Let $b_m\in\C$ be such that $\exp(2b_m)=\frac{a_m}{\ol{G}(y_m)F(y_m)}$. Observe that 
$$
\exp(2\ol{b_m})=\ol{\exp(2b_m)}=\frac{\ol{a_m}}{\ol{G}(\sigma(y_m))F(\sigma(y_m))}.
$$
If $\sigma(y_m)=y_m$, then $F(y_m)=\ol{F}(y_m)$ and $G(y_m)=\ol{G}(y_m)$, so
$$
\frac{\ol{a_m}}{\ol{G}(\sigma(y_m))F(\sigma(y_m))}=\frac{a_m}{G(y_m)F(y_m)}\in\R
$$
and we may assume $b_m\in\R$. By Cartan's Theorem B there exists $B\in\an(Y)$ such that $B(y_m)=b_m$ and $B(\sigma(y_m))=\ol{b_m}$ for each $m\geq1$. If we substitute $B$ by $\Re(B)=\frac{B+\ol{B\circ\sigma}}{2}$, we may assume that $B$ is invariant and $B(y_m)=b_m$ for each $m\geq1$. Thus, $\exp(B)$ is invariant and $(\exp(B)^2\ol{G}F)(y_m)=a_m$ for each $m\geq1$, so it is enough to take $A:=\exp(B)$.
\end{proof}

\begin{thm}\label{movc}
Let $F\in\an(Y)$ be an invariant function such that $F|_X$ is positive semidefinite and its zero-set is a discrete set. Then, after shrinking $Y$ if necessary, there exists $G\in\an(Y)$ such that the restriction of $F-G^2$ to $X$ is positive semidefinite, $\{F-G^2=0\}\cap (X\cup C)$ is a discrete set and $\{F-G^2=0\}\cap X=\{F=0\}\cap X$.
\end{thm}
\begin{proof}
Write $\{F=0\}\cap X=\{x_k\}_{k\geq1}$. As $(Y,\an_Y)$ is a Stein space, there exist invariant holomorphic functions $G_{k1},\ldots,G_{km_k}\in\an(Y)$ such that the maximal ideal $\gtm_{x_k}$ of $\an_{Y,x_k}$ is generated by the invariant holomorphic function germs $G_{k1,x_k},\ldots,G_{km_k,x_k}$ and $\{G_{k1}=0,\ldots,G_{km_k}=0\}=\{x_k\}$. The zero-set of $G_{k1}^2+\cdots+G_{km_k}^2$ in $X$ is $\{x_k\}$. By Lojasiewicz's inequality there exist an integer $s_k\geq1$ and a compact neighborhood $K_k\subset X$ of $x_k$ such that
$$
(G_{k1}^2+\cdots+G_{km_k}^2)^{4s_k}<F
$$ 
on $K_k\setminus\{x_k\}$. If $\lambda:=(\lambda_1,\ldots,\lambda_{m_k})\in(0,1)^{m_k}$, then
$$
(\lambda_1G_{k1}^2+\cdots+\lambda_{m_k}G_{km_k}^2)^{4s_k}<(G_{k1}^2+\cdots+G_{km_k}^2)^{4s_k}<F
$$ 
on $K_k\setminus\{x_k\}$.

Let $\{C_i\}_{i\geq1}$ be the collection of the irreducible components of $C$ as an analytic set of real dimension $1$. Pick points $y_i\in C_i\setminus X\cup\bigcup_{i\neq\ell}C_\ell$. As the family $\{C_i\}_i$ is locally finite, $\{y_i\}_{i\geq1}$ is a discrete set. For each $i,k\geq1$ consider 
$$
H_{ik}:\ \x_1G_{k1}^2(y_i)+\cdots+\x_{m_k}G_{km_k}^2(y_i)=0.
$$
The previous linear equation is not identically zero because $\{G_{k1}=0,\ldots,G_{km_k}=0\}=\{x_k\}$, so $H_{ik}$ is a hyperplane of $\R^{m_k}$. Thus, $(0,1)^{m_k}\setminus\bigcup_{j=1}H_{ik}$ is not empty and we chose $\lambda:=(\lambda_1,\ldots,\lambda_{m_k})\in(0,1)^{m_k}\setminus\bigcup_{i\geq1}H_{ik}$. Define the invariant holomorphic function
$$
G_k:=(\lambda_1G_{k1}^2+\cdots+\lambda_{m_k}G_{km_k}^2)^{s_k}.
$$ 
Observe that $D_k:=(\{G_k=0\}\cap C)\setminus\{x_k\}$ is a discrete set (because $G_k(y_i)\neq0$, so $G_k|_{C_i}\neq0$ for each $i\geq1$) and $\{G_k=0\}\cap X=\{x_k\}$. As $D_k\cup\{x_k\}$ is a discrete set, there exist open neighborhoods $U_k\subset Y$ of $x_k$ and $V_k\subset Y$ of $D_k$ such that $\cl(U_k)\cap \cl(V_k)=\varnothing$. Observe that $\cl(\{G_k=0\}\cap U_k)\cap D_k=\varnothing$, so $\cl(\{G_k=0\}\cap U_k)\cap C\subset\{x_k\}$. 

As $\{x_k\}_k\subset X$ is a discrete set, we may assume in addition $U_k$ is invariant, $\cl(U_k)\cap\cl(U_{k'})=\varnothing$ if $k\neq k'$ and $\{\cl(U_k)\}_k$ is locally finite, so $\bigcup_{k\geq1}\cl(U_k)$ is an invariant closed set. Thus, the family of invariant closed sets $\{\cl(\{G_k=0\}\cap U_k)\setminus(\{G_k=0\}\cap U_k)\}_k$ is also locally finite and its union is an invariant closed subset of $Y$. Consider the invariant open set $\Omega:=Y\setminus\bigcup_{k\geq1}(\cl(\{G_k=0\}\cap U_k)\setminus(\{G_k=0\}\cap U_k))$ and observe: \em each $\{G_k=0\}\cap U_k$ is an invariant closed subset of $\Omega$\em.

Thus, after shrinking $Y$ if necessary, $Z:=\bigcup_k(\{G_k=0\}\cap U_k)$ is an invariant closed subset of $Y$. Define
$$
\Ff_y:=\begin{cases}
G_{k,y}\an_{Y,y}&\text{if $y\in U_k\cap U$},\\
\an_{Y,y}&\text{if $y\in Y\setminus Z$},
\end{cases}
$$
which is a coherent sheaf of ideals whose zero-set is $Z$. By \cite{coen} the previous coherent sheaf is generated by $d+1$ global sections where $d:=\dim(X)$. Taking real and imaginary parts of such global sections, we deduce that $\Ff$ is generated by $2d+2$ invariant global sections $G_1',\ldots,G_{2d+2}'$. Observe that $G'_{j,x_k}=G_{k,x_k}A_{j,x_k}$ where $A_{j,x_k}\in\an_{Y,x_k}$ is invariant for each $j,k\geq1$.

In addition, $\{G_1'=0,\ldots,G_{2d+1}'=0\}\cap X=\{x_k\}_k$. Let $\mu_1,\ldots,\mu_{2d+2}\in(0,1)$ be such that 
$$
G':=\sum_{j=1}^{2d+2}\mu_jG_j'^2
$$ 
does not vanish at the points $\{y_i\}_i$. Thus, $G'$ does not vanish identically on any of the irreducible components $C_i$ of $C$ and $G'_{x_k}=G_{k,x_k}^2A'_{x_k}$ where $A'_{x_k}\in\an_{Y,x_k}$ is an invariant holomorphic unit such that $A'_{x_k}(x_k)>0$. By Lemma \ref{red0} there exists an invariant holomorphic function $A\in\an(Y)$ with empty zero-set such that $A(x_k)=\frac{1}{2A'_{x_k}(x_k)}$ for each $k\geq1$. Thus, $AG'$ satisfies:
\begin{itemize}
\item $AG'\geq0$ on $X$.
\item $(AG')^2<\frac{1}{2}G_k^4<F$ on $K_k'\setminus\{x_k\}$ where $K_k'\subset K_k$ is a compact neighborhood of $x_k$ in $X$.
\end{itemize}
Let $\sigma:X\to[0,1]$ be a continuous function such that $\sigma(x_k)=1$ for each $k\geq1$ and its support is contained in $\bigcup_{k\geq1}K_k'$. Observe that $\veps:=(F|_X+\sigma)\exp(-2-F^2)$ is a strictly positive continuous function on $X$ that takes values in the interval $(0,1)$. By Theorem \ref{fa} there exists an invariant homomorphic function $H:Y\to\C$ such that $|H|_X-\frac{\veps}{2}|<\frac{\veps}{4}$ and $H(y_i)\neq0$ for each $i\geq1$, so $\frac{\veps}{4}<H|_X<\frac{3\veps}{4}<\frac{3}{4}$. Define $G'':=AG'\exp(-1-AG')H$ and let us check: \em There exists $0<\lambda<1$ such that $G:=\lambda G''$ satisfies the required properties\em. 

Observe that $G''(y_i)\neq0$ and pick $\lambda\in(0,1)$ different from $\sqrt{F(y_i)}/G''(y_i)$ for each $i\geq1$. Thus, $(F-(\lambda G'')^2)(y_i)\neq0$ for each $i\geq1$ and $\{F-(\lambda G'')^2=0\}\cap C$ is a discrete set. We prove next $F-(\lambda G'')^2>0$ on $X\setminus\{x_k\}_{k\geq1}$. On $K_k'\setminus\{x_k\}$ it holds
$$
(\lambda G'')^2<(AG'\exp(-1-AG')H)^2<(AG')^2<\frac{1}{2}G_k^4<F
$$
for each $k\geq1$. On $X\setminus\bigcup_{k\geq1}K_k'$ we have $\sigma=0$ and
$$
0<H<\frac{3\veps}{4}<(F+\sigma)\exp(-2-F^2)=F\exp(-2-F^2)<\frac{F}{2+F^2}.
$$
Thus, on $X\setminus\bigcup_{k\geq1}K_k'$ it holds 
$$
(\lambda G'')^2<G''^2<H^2<\Big(\frac{F}{2+F^2}\Big)^2<\frac{F}{2+F^2}<F,
$$
as required.
\end{proof}

\subsection{Eliminating isolated zeros from the imaginary part}
We will deal with normal Stein spaces $(Y,\an_Y)$ endowed with an anti-involution $\sigma:Y\to Y$. We denote the fixed space of $\sigma$ with $X$ and let $C\subset Y$ be an invariant analytic set of real dimension $1$ such that $X\cap C$ is a discrete set. Let $\{C_i\}_{i\geq1}$ be the irreducible components of $C$ and $\alpha_i:\R\to C_i$ a surjective analytic map.

Our purpose is to represent the invariant holomorphic functions $F\in\an(Y)$ such that $F|_X\geq0$ and $\{F=0\}\cap(X\cup C)$ is a discrete set as sum of squares of meromorphic functions. We provide an additional reduction to assume $\{F=0\}\cap C\subset X\cap C$. We denote the set of all the invariant holomorphic functions on $Y$ that are sums of $k$ squares of invariant holomorphic functions on $Y$ with $\Sos_k\an^\sigma(Y)^2$. We will use freely that $\Sos_k\an^\sigma(Y)^2$ is multiplicatively closed if $k=1,2,4,8$.

\begin{thm}\label{maintool}
Let $F:Y\to\C$ be an invariant holomorphic function such that $F|_X\geq0$ and $F|_{C_i}\neq0$ for each $i\geq1$. Then, after shrinking $Y$ if necessary, there exist invariant holomorphic functions $F',H_1,H_2:Y\to\C$ such that:
\begin{itemize}
\item[(i)] $H_1,H_2\in\Sos_4\an^\sigma(Y)^2$.
\item[(ii)] $\{F'=0\}\cap C\subset X\cap C$ and $\{H_\ell=0\}\cap X=\varnothing$ for $\ell=1,2$.
\item[(iii)] $FH_1=F'H_2$ and $F'_x\an_{Y,x}=F_x\an_{Y,x}$ for each $x\in X$.
\end{itemize}
\end{thm}

Before we prove the previous result we need some preliminary lemmas.

\begin{lem}\label{red1}
Let $Z\subset Y$ be an analytic subset of real dimension $\leq 1$ and $G_1,G_2,G_3:Y\to\C$ holomorphic functions such that: 
\begin{itemize}
\item[(i)] $\{G_1=0,G_2=0,G_3=0\}$ does not contain any of the irreducible components of $Z$.
\item[(ii)] For each $z\in Z$ there exists $k(z)\in\{1,2,3\}$ such that $(G_1,G_2,G_3)\an_{Y,z}=(G_{k(z)})\an_{Y,z}$.
\end{itemize}
Then there exists $\lambda_1,\lambda_2,\lambda_3\in\C$ such that $(\lambda_1G_1+\lambda_2G_2+\lambda_3G_3)_z$ divides $G_{j,z}$ for each $z\in Z$ and each $j=1,2,3$. 
\end{lem}
\begin{proof}
Write $Z:=D\cup\bigcup_{i\geq1}Z_i$, where the $Z_i$ are the irreducible components of $Z$ of dimension $1$ and $D$ is the union of the irreducible components of $Z$ of dimension $0$. Note that $D$ is a discrete subset of $Y$. For each $i\geq1$ let $\beta_i:\R\to Z_i$ be a surjective analytic map (which one finds using the normalization). Denote $D':=D\cap\{G_1=0,G_2=0,G_3=0\}$.

For each $p\in D\setminus D'$ either $G_1(p)\neq 0$ or $G_2(p)\neq 0$ or $G_3(p)\neq0$. Define 
$$
\Tt_p:=\{(\lambda_1,\lambda_2,\lambda_3)\in\C^2:\, G_1(p)+\lambda_2G_2(p)+\lambda_3G_3(p)=0\}, 
$$
which is a hyperplane of $\C^3$ through the origin.

The set $D'':=\{G_1=0,G_2=0,G_3=0\}\cap\bigcup_{i\geq1}Z_i$ is discrete. Fix $i\geq1$ and let $E_i:=\beta_i^{-1}(D'')$, which is a discrete subset of $\R$. Define $f_{ij}:=G_j\circ\beta_i|_{\R\setminus E_i}$ for $j=1,2,3$ and the non-zero analytic equation
$$
f_{i1}(t)\x_1+f_{i2}(t)\x_2+f_{i3}(t)\x_3=0
$$
for $t\in\R\setminus E_i$. Consider the analytic maps
$$
\begin{cases}
\varphi_{1i}:\C^2\times(\R\setminus E_i)\to\C^3,\ (z_1,z_2,t)\mapsto z_1(f_{3i}(t),0,-f_{1i}(t))+z_2(f_{2i}(t),-f_{1i}(t),0),\\
\varphi_{2i}:\C^2\times(\R\setminus E_i)\to\C^3,\ (z_1,z_2,t)\mapsto z_1(0,f_{3i}(t),-f_{2i}(t))+z_2(f_{2i}(t),-f_{1i}(t),0),\\
\varphi_{3i}:\C^2\times(\R\setminus E_i)\to\C^3,\ (z_1,z_2,t)\mapsto z_1(f_{3i}(t),0,-f_{1i}(t))+z_2(0,f_{3i}(t),-f_{2i}(t)).
\end{cases}
$$
Observe that 
$$
\Ss_i:=\bigcup_{t\in\R\setminus E_i}\{(x_1,x_2,x_3)\in\C^3:\ f_{i1}(t)\x_1+f_{i2}(t)\x_2+f_{i3}(t)\x_3=0\}=\bigcup_{j=1}^3\im(\varphi_{ji}).
$$
By Sard's theorem the set $\C^3\setminus\Ss_i$ is residual. 

If $z\in D'\cup D''$, either $G_{1,z}$ divides $G_{2,z},G_{3,z}$ or ${G_2,z}$ divides $G_{1,z},G_{3,z}$ or ${G_3,z}$ divides $G_{1,z},G_{2,z}$. Thus, there exists $\eta_z,\xi_z\in\an_{Y,z}$ such that 
$$
\begin{cases}
G_{2,z}=G_{1,z}\eta_z, G_{3,z}=G_{1,z}\xi_z&\text{(Case 1) or}\\
G_{1,z}=G_{2,z}\eta_z, G_{3,z}=G_{2,z}\xi_z&\text{(Case 2) or}\\
G_{1,z}=G_{3,z}\eta_z, G_{2,z}=G_{3,z}\xi_z&\text{(Case 3)}.
\end{cases}
$$
Consequently,
$$
\lambda_1G_{1,z}+\lambda_2G_{2,z}+\lambda_3G_{3,z}=\begin{cases}
G_{1,z}(\lambda_1+\lambda_2\eta_z+\lambda_3\xi_z)&\text{(Case 1) or}\\
G_{2,z}(\lambda_1\eta_z+\lambda_2+\lambda_3\xi_z)&\text{(Case 2) or}\\
G_{3,z}(\lambda_1\eta_z+\lambda_2\xi_z+\lambda_3)&\text{(Case 3)}.
\end{cases}
$$
There exists a discrete set $\Delta\subset\C^3$ such that if $\lambda:=(\lambda_1,\lambda_2,\lambda_3)\not\in\Delta$,
$$
\gamma_z:=\begin{cases}
\lambda_1+\lambda_2\eta_z+\lambda_3\xi_z&\text{(Case 1) or}\\
\lambda_1\eta_z+\lambda_2+\lambda_3\xi_z&\text{(Case 2) or}\\
\lambda_1\eta_z+\lambda_2\xi_z+\lambda_3&\text{(Case 3).}
\end{cases}
$$
is a unit for each $z\in D'\cup D''$. The set $\Ss:=\C^3\setminus(\bigcup_{p\in D'}\Tt_p\cup\bigcup_{i\geq1}\Ss_i\cup \Delta)$ is residual and we claim: \em If $\lambda\in\Ss$, then $G_\lambda:=\lambda_1G_{1,z}+\lambda_2G_{2,z}+\lambda_3G_{3,z}$ divides $G_{1,z},G_{2,z},G_{3,z}$ for each $z\in Z$\em. We distinguish several cases: 

\noindent{\sc Case 1}. If $z\in D\setminus D'$, then $G_\lambda(z)\neq0$ (because $\lambda\not\in\bigcup_{p\in D\setminus D'}\Tt_p$), so $G_{\lambda,z}$ divides $G_{1,z},G_{2,z},G_{3,z}$.

\noindent{\sc Case 2}. If $z\in Z\setminus(D\cup D'')$, then $G_\lambda(z)\neq0$ (because $\lambda\not\in\bigcup_{i\geq1}\Ss_i$), so $G_{\lambda,z}$ divides $G_{1,z},G_{2,z},G_{3,z}$.

\noindent{\sc Case 3}. If $z\in D'\cup D''$, then $G_{\lambda,z}=G_{j,z}\gamma_z$ for some $j=1,2,3$ such that $G_{j,z}$ divides the remaining $G_{k,z}$ and $\gamma_z\in\an_{Y,z}$ is a unit (because $\lambda\not\in\Delta$), so $G_{\lambda,z}$ divides $G_{1,z},G_{2,z},G_{3,z}$, as required.
\end{proof}

\begin{lem}\label{red4}
Let $F_1,F_2:Y\to\C$ be invariant holomorphic functions such that $F_j|_{C_i}\neq0$ for $j=1,2$ and $i\geq1$ and $F_{1,z}$ divides $F_{2,z}$ for each $z\in C$. Then there exists an invariant holomorphic function $E\in\an(Y)$ with empty zero-set such that $H:=F_1^2+E^2F_2^2$ satisfies: $H_z$ divides $F_{1,z}^2$ for each $z\in C$.
\end{lem}
\begin{proof}
Fix $i\geq1$ and let $z_i\in C_i\setminus(X\cup\{F_1=0\}\cup\{F_2=0\})$. By Lemma \ref{red0} there exists an invariant holomorphic function $E_0\in\an(Y)$ with empty zero-set such that $(\ol{F_1}F_2E_0)(z_i)\in\C\setminus(\R\cup\sqrt{-1}\R)$ for each $i\geq1$. Substitute $F_2$ by $E_0F_2$. As $F_{1,z}$ divides $F_{2,z}$ for each $z\in C$, there exist an open neighborhood $V$ of $C$ and a holomorphic function $A\in\an(V)$ such that $F_2=F_1A$ on $V$. As $\ol{F_1}F_2=\ol{F_1}F_1A$ and $(\ol{F_1}F_2)(z_i)\in\C\setminus(\R\cup\sqrt{-1}\R)$ for each $i\geq1$, we deduce $A(z_i)\in\C\setminus(\R\cup\sqrt{-1}\R)$ for each $i\geq1$. Fix $\lambda\in(0,+\infty)$ and write
$$
(F_1^2+\lambda F_2^2)\circ\alpha_i=(F_1^2\circ\alpha_i)(1+\lambda(A^2\circ\alpha_i)).
$$
Write $A\circ\alpha_i:=a_i+\sqrt{-1}b_i$ where $a_i,b_i\in\an(\R)\setminus\{0\}$ for $i\geq1$, so the set $\Ss_i:=\{a_ib_i=0\}$ is discrete for each $i\geq1$. Consider the linear system $1+\z(a_i+\sqrt{-1}b_i)^2=0$, which we rewrite as
\begin{equation}\label{system}
\begin{cases}
1+\z(a_i^2-b_i^2)=0,\\
\z a_ib_i=0.
\end{cases}
\end{equation}
If $(a_ib_i)(t)\neq0$, then $(1+\lambda(A\circ\alpha_i))(t)\neq0$ for each $\lambda\in\R$. Define 
$$
\Tt_i:=\{-1/(a_i^2-b_i^2)(t):\ t\in \Ss_i,\ (a_i^2-b_i^2)(t)\neq0\},
$$
which is a countable set. Let $\lambda\in\R\setminus\bigcup_{i\geq1}\Tt_i$ and take $H:=F_1^2+\lambda F_2^2$. We claim:\em 
$$
(1+\lambda A^2)\circ\alpha_i=1+\lambda(a_i^2-b_i^2)+\sqrt{-1}a_ib_i
$$
is nowhere zero for each $i\geq1$\em.

It is enough to analyze the values $t\in \Ss_i$ for each $i\geq1$. We have
$$
(1+\lambda(a_i^2-b_i^2)+\sqrt{-1}a_ib_i)(t)=(1+\lambda(a_i^2-b_i^2))(t)\begin{cases}
\neq0&\text{if $(a_i^2-b_i^2)(t)\neq0$},\\
=1\neq0&\text{if $(a_i^2-b_i^2)(t)=0$},
\end{cases}
$$
as required.
\end{proof}

\begin{lem}\label{red5}
Let $F,G:Y\to\C$ be invariant holomorphic functions such that $F|_{C_i},G|_{C_i}\neq0$ for each $i\geq1$ and $F_z\an_{Y,z}=G_z\an_{Y,z}$ for each $z\in C\setminus X$. Then there exists an invariant holomorphic function $E\in\an(Y)$ with empty zero-set such that $H:=FG+E^2F^3$ satisfies: $H_z$ divides $F_z^2$ for each $z\in C\setminus X$.
\end{lem}
\begin{proof}
Fix $i\geq1$ and let $z_i\in C_i\setminus (X\cup\{F=0\}\cup\{G=0\})$. By Lemma \ref{red0} there exists an invariant holomorphic function $E_0\in\an(Y)$ with empty zero-set such that $(\ol{G}F^2E_0^2)(z_i)\in\C\setminus(\R\cup\sqrt{-1}\R)$ for each $i\geq1$. As $F_z\an_{Y,z}=G_z\an_{Y,z}$ for each $z\in C\setminus X$ there exists an open neighborhood $V\subset Y$ of $C\setminus X$ and $A\in\an(V)$ such that $G=AF$ and $\{A=0\}=\varnothing$. It holds 
$$
|F(z_i)|^2(\ol{A}FE_0^2)(z_i)=(\ol{A}\ol{F}F^2E_0^2)(z_i)=(\ol{G}F^2E_0^2)(z_i)\in\C\setminus(\R\cup\sqrt{-1}\R),
$$
so $(\ol{A}FE_0^2)(z_i)\in\C\setminus(\R\cup\sqrt{-1}\R)$ for each $i\geq1$.

For each $\lambda\in(0,+\infty)$ we consider $H:=FG+\lambda E_0^2F^3$. We have $
H|_V=F^2|_V(A+\lambda E_0^2|_VF|_V)$. Let us find $\lambda\in(0,+\infty)$ such that $B:=A+\lambda E_0^2|_VF|_V$ does not vanish on $C\setminus X$. Write $A\circ\alpha_i:=a_i+\sqrt{-1}b_i$ and $(E_0^2F)\circ\alpha_i=f_i+\sqrt{-1}g_i$ where $a_i,b_i,f_i,g_i\an(\R)$ and $f_i,g_i\neq0$. We have 
$$
B\circ\alpha_i=a_i+\sqrt{-1}b_i+\lambda(f_i+\sqrt{-1}g_i)=(a_i+\lambda f_i)+\sqrt{-1}(b_i+\lambda g_i).
$$
Consider the linear system $a_i+\sqrt{-1}b_i+\z(f_i+\sqrt{-1}g_i)=0$ that we rewrite as
\begin{equation}\label{system2}
\begin{cases}
a_i+\z f_i=0,\\
b_i+\z g_i=0.
\end{cases}
\end{equation}
The determinant of the previous linear system is $a_ig_i-b_if_i$, which is the imaginary part of $(\ol{A}FE_0^2)\circ\alpha_i$. As $(\ol{A}FE_0^2)(z_i)\in\C\setminus(\R\cup\sqrt{-1}\R)$, we deduce that $a_ig_i-b_if_i\neq0$. Let $\Ss_i:=\{a_ig_i-b_if_i=0\}$, which is a discrete set. If $t\in\R\setminus \Ss_i$, then the linear system \eqref{system2} is incompatible and $(B\circ\alpha_i)(t)\neq0$. If $t\in \Ss_i$, the linear system \eqref{system2} has rank $1$ (recall that $\{A=0\}=\varnothing$) and we define $\Tt_t$ as the maybe empty solution of the linear system \eqref{system2}. The set $\Ss:=\bigcup_{i\geq1}\bigcup_{t\in \Ss_i}\Tt_t$ is countable, so $(0,+\infty)\setminus\Ss\neq\varnothing$. If $\lambda\in(0,+\infty)\setminus\Ss$ and $t\in \Ss_i$, we have $\lambda\not\in\Tt_t$ and
$$
((A+\lambda E_0^2F)\circ\alpha_i)(t)=(a_i+\sqrt{-1}b_i)(t)+\lambda(f_i+\sqrt{-1}g_i)(t)\neq0.
$$
Thus, $\{A+\lambda E_0^2F=0\}\cap(C\setminus X)=\varnothing$ and $H_z$ divides $F_z^2$ for each $z\in C\setminus X$, as required.
\end{proof}

We are ready to prove Theorem \ref{maintool}.

\begin{proof}[Proof of Theorem \em\ref{maintool}]
The proof is conducted in several steps:

\noindent{\sc Step 1.} Denote $D:=X\cap C$. \em We may shrink $Y$ in such a way that the irreducible components of $\{F=0\}$ either meet $X$ or meet $C\setminus D$\em.

Define $D':=(\{F=0\}\cap C)\setminus D$, which is a discrete set because $F|_{C_i}\neq0$ for each $i\geq1$. As $X,D'$ are closed and disjoint sets, there exist disjoint open neighborhoods $W_1,V_1\subset Y$ of $X,D'$. For each $x\in C\setminus (D\cup D')$ we choose an open neighborhood $V^x$ that does not meet $\{F=0\}$. Define $\Omega_0:=W_1\cup V_1\cup \bigcup_{x\in C\setminus(D\cup D')}V^x$ and let $\Ss$ be an irreducible component of $\{F=0\}\cap \Omega_0$. Then
$$
\Ss=(\Ss\cap W_1)\cup(\Ss\cap V_1)\cup \bigcup_{x\in C\setminus(D\cup D')}(\Ss\cap V^x)=(\Ss\cap W_1)\cup(\Ss\cap V_1).
$$
As $W_1\cap V_1=\varnothing$ and $\Ss$ is irreducible, either $\Ss\cap W_1=\varnothing$ or $\Ss\cap V_1=\varnothing$. Thus, either $\Ss$ meets $C\setminus D$ or it meets $X$.

Let $\Omega_1\subset\Omega_0$ be an invariant open Stein neighborhood of $X\cup C$ in $Y$ and observe that it satisfies the required properties. Thus, it is enough to substitute $Y$ by $\Omega_1$.

\noindent{\sc Step 2.} \em Construction of invariant holomorphic functions $G,H_2:Y\to\C$ such that: 
\begin{itemize}
\item[(1)] $G\in\Sos_2\an^\sigma(Y)^2$, $\{G=0\}\cap C=D'$ and $\{G=0\}\cap \{F=0\}\cap X=\varnothing$.
\item[(2)] $G_z\an_{Y,z}=F_z\an_{Y,z}$ for each $z\in C\setminus D$.
\item[(3)] $H_2\in\Sos_2\an^\sigma(Y)^2$ and $H_{2,z}\an_{Y,z}=F^2_z\an_{Y,z}$ for each $z\in C\setminus D$.
\item[(4)] $\{H_2=0\}\cap(X\cup C)=\{F=0\}\cap (C\setminus D)=D'$.
\end{itemize}
\em

For each $y\in D'$ we have $\sigma(y)\neq y$. We write $D'=D'_0\cup\sigma(D'_0)$ where $D'_0\cap \sigma(D'_0)=\varnothing$. For each $y\in D'_0$ we denote the connected component of $\{F=0\}\cap Y$ that contains $y$ with $\Ss_y$. Shrinking $Y$ if necessary, we may assume the family $\{\Ss_y\}_{y\in D'_0}$ is locally finite and $\Ss_y\cap \Ss_{y'}=\varnothing$ if $y,y'\in D'_0$ and $y\neq y'$. Consider the coherent sheaf of ideals of $\an_Y$:
$$
\Jj_z:=\begin{cases}
F_z\an_{Y,z}&\text{if $z\in \bigcup_{y\in D'_0}\Ss_y$},\\
\an_{Y,z}&\text{otherwise}.
\end{cases}
$$
By \cite{coen} $\Jj$ is globally generated by three global sections $G_1,G_2,G_3:Y\to\C$. Define $Z:=C\cup(\{F=0\}\cap X)$, which is an invariant analytic set of real dimension $1$. For each $z\in Z$ there exists $k(z)\in\{1,2,3\}$ such that $(G_1,G_2,G_3)\an_{Y,z}=(G_{k(z)})\an_{Y,z}$. By Lemma \ref{red1} applied to $Z$ and $G_1,G_2,G_3$, there exists a global section $G_4$ of $\Jj$ such that $G_{4,z}$ divides $G_{1,z}, G_{2,z}, G_{3,z}$ for each $z\in Z$. In particular, $\Jj_z=G_4\an_{Y,z}$ for each $z\in Z$ and $\{G_4=0\}\cap\{F=0\}\cap X=\varnothing$. 

Define $G:=G_4\ol{G_4\circ\sigma}\in\Sos_2\an^\sigma(Y)^2$, which is invariant. As $C$ is invariant and $\{G_4=0\}\cap C=D_0'$, we have $\{\ol{G_4\circ\sigma}=0\}\cap C=\sigma(D_0')$, so 
$$
\{G=0\}\cap C=(\{G_4=0\}\cap C)\cup(\{\ol{G_4\circ\sigma}=0\}\cap C)=D'_0\cup\sigma(D'_0)=D'.
$$
In addition, $G_z\an_{Y,z}=F_z\an_{Y,z}$ for each $z\in D'$. This is because $\Jj_z=\an_{Y,z}$ for each $z\in\sigma(D'_0)$, so $\ol{G_4\circ\sigma}(z)\neq0$ for each $z\in D_0'$, and $\Jj_z=G_{4,z}\an_{Y,z}$ for each $z\in D_0'$.

If $x\in\{F=0\}\cap X$, we have $G(x)=(G_4\ol{G_4\circ\sigma})(x)=|G_4(x)|^2\neq 0$ because $\{G_4=0\}\cap \{F=0\}\cap X=\varnothing$. Thus, $G_z\an_{Y,z}=F_z\an_{Y,z}$ for each $z\in C\setminus D$ and $\{G=0\}\cap \{F=0\}\cap X=\varnothing$, so $G$ satisfies the properties (1) and (2).

Pick $z_i\in C_i\setminus(D\cup D')$ for each $i\geq1$. Let $\lambda_2,\lambda_3\in(0,+\infty)$ be such that $H:=G_1\ol{G_1\circ\sigma}+\lambda_2G_2\ol{G_2\circ\sigma}+\lambda_3G_3\ol{G_3\circ\sigma}$ satisfies $H(z_i)\neq0$ for each $i\geq1$. Thus, $H|_{C_i}\neq0$ for each $i\geq1$.

As $G_j\ol{G_j\circ\sigma}\in\Sos_2\an^\sigma(Y)^2$ is invariant, $\{H=0\}\cap X=\{G_j\ol{G_j\circ\sigma}=0, j=1,2,3\}\cap X$. As $\Jj_z=\an_{Y,z}$ for each $z\in X$, we deduce that $\{H=0\}\cap X=\varnothing$. As $G_z$ divides $H_z$ for each $z\in C$, there exists by Lemma \ref{red4} an invariant holomorphic function $E_2\in\an(Y)$ such that the invariant holomorphic function
$$
H_2:=G^2+E_2^2H^2\in\Sos_2\an^\sigma(Y)^2
$$
has the following property: \em $H_{2,z}$ divides $G_z^2$ for each $z\in C$\em, so $H_{2,z}\an_{Y,z}=G_z^2\an_{Y,z}$ for each $z\in C$ and $H_{2,z}\an_{Y,z}=G_z^2\an_{Y,z}=F_z^2\an_{Y,z}$ for each $z\in C\setminus D$. In particular, $\{H_2=0\}\cap(C\setminus D)=D'$. In addition,
$$
\{H_2=0\}\cap X=\{G=0,H=0\}\cap X=\{H=0\}\cap X=\varnothing,
$$
so $H_2$ satisfies the properties (3) and (4).

\noindent{\sc Step 3.} \em Construction of an invariant holomorphic function $H_1\in\Sos_3\an^\sigma(Y)^2$ such that: 
\begin{itemize}
\item[(1)] $\{H_1=0\}\cap(X\cup C)=\{H_1=0\}\cap C=D'$,
\item[(2)] $(FH_1)_x\an_{Y,x}=F_x\an_{Y,x}$ for each $x\in X$ and
\item[(3)] $(FH_1)_z\an_{Y,z}=F_z^2\an_{Y,z}$ for each $z\in C\setminus D$.
\end{itemize}
\em

By Lemma \ref{red5} there exists an invariant holomorphic function $E_1\in\an(Y)$ with empty zero-set such that the invariant holomorphic function $H_1:=G+E_1^2F^2\in\Sos_3\an^\sigma(Y)^2$ satisfies: $(FH_1)_z$ divides $F_z^2$ for each $z\in C\setminus D$, so $(FH_1)_z\an_{Y,z}=F_z^2\an_{Y,z}$ for each $z\in C\setminus D$ (because $G_z\an_{Y,z}=F_z\an_{Y,z}$ for each $z\in C\setminus D$). As $\{G=0\}\cap\{F=0\}\cap X=\varnothing$, we have $\{H_1=0\}\cap X=\{G=0,F=0\}\cap X=\varnothing$, so $(FH_1)_x\an_{Y,x}=F_x\an_{Y,x}$ for each $x\in X$. In particular, $\{H_1=0\}\cap C=D'$.

\noindent{\sc Step 4.} \em After shrinking $Y$, the invariant function $F':=\frac{FH_1}{H_2}$ is holomorphic, $\{F'=0\}\cap C\subset D$ and $F'_x\an_{Y,x}=F_x\an_{Y,x}$ for each $x\in X$\em.

Let us check: \em For each $z\in X\cup C$ the germ $F'_z$ is holomorphic\em. 

By {\sc Step 2} (3) we have $\{H_2=0\}\cap(X\cup C)=D'$, so $F'_z$ is a holomorphic germ for each $z\in X\cup (C\setminus D')$. If $z\in D'\subset C\setminus D$, we get:
\begin{itemize}
\item $H_{2,z}\an_{Y,z}=F^2_z\an_{Y,z}$ for each $z\in C\setminus D$ ({\sc Step 2} (3)).
\item $(FH_1)_z\an_{Y,z}=F_z^2\an_{Y,z}$ for each $z\in C\setminus D$ ({\sc Step 3} (3)).
\end{itemize} 
Thus, $H_{2,z}\an_{Y,z}=(FH_1)_z\an_{Y,z}$ for each $z\in C\setminus D$. Consequently, $F'_z$ is holomorphic for each $z\in C\setminus D$ and $F'_z$ is a unit for each $z\in C\setminus D$, so $\{F'=0\}\cap C\subset D$.

As $(FH_1)_x\an_{Y,x}=F_x\an_{Y,x}$ for each $x\in X$ ({\sc Step 3} (2)) and $H_2(x)\neq 0$ for each $x\in X$ ({\sc Step 2} (4)), we obtain $F'_x\an_{Y,x}=F_x\an_{Y,x}$ for each $x\in X$, as required.
\end{proof}

\subsection{Eliminating isolated zeros from the real part space}
Our next purpose is to reduce our problem to represent invariant holomorphic functions on $Y$ with empty zero-set whose restriction to $X$ is strictly positive as sum of squares of meromorphic functions.

\begin{thm}\label{discretezero-set}
Let $F\in\an(Y)$ be an invariant holomorphic function such that $F|_{X}\geq 0$ and $\{F=0\}\cap(X\cup C)$ is a discrete subset of $X$. Then, after shrinking $Y$ if necessary, there exist invariant holomorphic functions $H,G':Y\to\C$ such that $\{H=0\}=\varnothing$, $H|_X$ is strictly positive, $\{G'=0\}\cap X\subset\{F=0\}\cap X$ and $G'^2HF\in\Sos_8\an^\sigma(Y)^2$.
\end{thm}

Before proving the previous result, we need a preliminary lemma.

\begin{lem}\label{extension}
Let $\Phi:Y\to\C$ be an invariant holomorphic function. Let $U$ be an invariant open neighborhood of the connected components of $\Phi^{-1}(0)$ that meet $X\cup C$ and suppose that $U$ does not meet the other connected components of $\Phi^{-1}(0)$. For each invariant holomorphic function $B:U\to\C$ there exists an invariant holomorphic function $A:Y\to\C$ such that $\Phi|_U$ divides $A|_U-B$.
\end{lem}
\begin{proof}
Consider the coherent sheaf of ideals $\Jj\subset\an_Y$ of $\an_Y$ generated by $\Phi$ and the exact sequence of coherent sheafs $0\rightarrow\Jj\rightarrow\an_Y\rightarrow\an_Y/\Jj\rightarrow0$. We have a corresponding diagram of cross sections:
$$
\xymatrix{
\Jj(Y)\ar[r]\ar[d]&\an(Y)\ar[r]\ar[d]&H^0(Y,\an_Y/\Jj)\ar[d]\\
\Jj(U)\ar[r]&\an(U)\ar[r]&H^0({U},\an_Y/\Jj)
}
$$
The upper right arrow is surjective by Cartan's Theorem B because $Y$ is a Stein space. The right vertical arrow is also surjective because each cross section of $\an_Y/\Jj$ on $U$ can be extended by zero to $Y$ because the support of $\an_Y/\Jj$ in $U$ is closed in $Y$. Thus, we have a linear surjective homomorphism $
\varphi:\an(Y)\longrightarrow\an(U)/\Jj(U)\equiv H^0({U},\an_Y/\Jj)$. Let $F\in\an(Y)$ be such that $\varphi(F)=B$. We have $F|_U-B\in\Jj(U)$, so there exists $\Lambda\in\an(U)$ such that $F|_U-B=\Lambda\Phi|_U$. Take $A=\Re(F)$ and observe that
$$
A|_U-B=\Re(F|_U)-B=\Re(F|_U-B)=\Re(\Lambda\Phi|_U)=\Re(\Lambda)\Phi|_U\in\Jj(U),
$$ 
as required.
\end{proof}

We are ready to prove Theorem \ref{discretezero-set}.

\begin{proof}[Proof of Theorem \em\ref{discretezero-set}]
The proof is conducted in several steps:

\noindent{\sc Step 0.} \em Initial preparation\em. For each $p\in\{F=0\}\cap X$ there exist by Theorem \ref{localcase} analytic function germs $h_{0,p},h_{1,p},h_{2,p},h_{3,p},h_{4,p}\in\an_{X,p}$ such that $h_{0,p}^2f_p=h_{1,p}^2+h_{2,p}+h_{3,p}+h_{4,p}^2$ and $h_{0,p}$ is a sum of squares in $\an_{X,p}$ with $\{h_{0,p}=0\}\subset\{f_p=0\}\subset\{p\}$. Multiplying the previous expression by $h_{0,p}^2$, we may assume that $h_{0,p}$ is a square. Thus, for each $p\in\{F=0\}\cap X$ there exist an invariant open neighborhood $U^p\subset Y$ and invariant holomorphic functions $H_{i,(p)}\in\an(U^p)$ such that $H_{0,(p)}^2F|_{U^p}=H_{1,(p)}^2+H_{2,(p)}^2+H_{3,(p)}^2+H_{4,(p)}^2$, $H_{0,(p)}$ is the square of an invariant holomorphic function and $\cl(\{H_{0,(p)}=0\})\cap X\subset\{p\}$. We need in addition that $\cl(\{H_{0,(p)}=0\})\cap C\subset\{p\}$. To that end it is enough to check: \em We may assume $H_{0,(p)}$ is not identically zero on any of the irreducible components of the germ of $C$ at $p$\em.

Identify $Y_p$ with an analytic model contained in $\C^n_p$ for some $n\geq1$. For each unitary vector $v\in\sph^n$ consider the holomorphic map $\Phi_v:U^p\to\C^n, x\mapsto x+vF^2(x)$. If $U^p$ is small enough, $\Phi_v$ is a holomorphic diffeomorphism onto its image and $F\circ\Phi_v=FE^2$ where $E\in\an(U^p)$ is an invariant holomorphic function such that $\{E=0\}=\varnothing$ (use for instance Taylor's expansion of $F((\x_1,\ldots,\x_n)+\z(\y_1,\ldots,\y_n))$ at $(\x_1,\ldots,\x_n)$ to check the latter condition). As $F|_{C\setminus X}$ does not vanish, we may choose $v\in\R^n$ such that $H_{0,(p)}\circ\Phi_v$ is not identically zero on any of the irreducible components of the germ of $C$ at $p$. Thus, it is enough to substitute $H_{0,(p)}$ by $(H_{0,(p)}\circ\Phi_v)E$ and $H_{i,(p)}$ by $H_{i,(p)}\circ\Phi_v$ for $i=1,2,3,4$.

We may assume in addition $\cl(U^p)\cap\cl(U^q)=\varnothing$ if $p\neq q$ and $\{\cl(U^p)\}_{p\in\{F=0\}\cap X}$ is a locally finite family. The union $T:=\bigcup_{p\in\{F=0\}\cap X}\cl(\{H_{0,(p)}=0\})\setminus U^p$ is a closed subset of $Y$ that does meet $X\cup C$. If we substitute $Y$ by an invariant open Stein neighborhood of $X\cup C$ contained in the open set $U:=Y\setminus T$, we may assume $\{H_{0,(p)}=0\}$ is a closed subset of $Y$ for each $p\in\{F=0\}\cap X$. 

\noindent{\sc Step 1.} \em Global denominator\em. Consider the locally principal coherent sheaf of ideals defined by:
$$
\Jj_z=\left\{
\begin{array}{ll}
H_{0,(p),z}\an_{Y,z}&\text{if $z\in U^p$}\\
\an_{Y,z}&\text{if $z\in Y\setminus\bigcup_{p\in\{F=0\}\cap X}U^p$}.
\end{array}
\right.
$$
By \cite{coen} there exist holomorphic functions $G_1,G_2,G_3\in\an(Y)$ that generate $\Jj$. As each $H_{0,(p),z}$ is invariant, $\ol{G_1\circ\sigma},\ol{G_2\circ\sigma},\ol{G_3\circ\sigma}\in H^0(Y,\Jj)$. Thus, we may assume there exist $6$ invariant holomorphic functions $G_1,\ldots,G_6\in\an(Y)$ that generate $\Jj$. We choose $\lambda_1,\ldots,\lambda_6\in(0,+\infty)$ such that $G_0:=\sum_{i=1}^6\lambda_iG_i^2$ is not identically zero on any of the irreducible components $C_i$ of $C$. By Theorem \ref{maintool} there exists (after shrinking $Y$ if necessary) an invariant holomorphic function $G\in\an(Y)$ such that $G|_X\geq0$, $\{G=0\}\cap C\subset X\cap C$ and $G_x\an_{Y,x}=G_{0,x}\an_{Y,x}$ for each $x\in X$. In particular, $\{G=0\}\cap(X\cup C)\subset\{F=0\}\cap X$.

Fix $p\in\{F=0\}\cap X$. After shrinking $U^p$ if necessary, there exists an invariant unit $E_p\in\an(U^p)$ such that $G|_{U^p}=H_{0,(p)}E_p$. Thus, on $U^p$ we have
$$
G^2F=E_p^2H_{0,(p)}^4F=\sum_{k=1}^4(E_pH_{0,(p)}H_{k,(p)})^2
$$
and we substitute $H_{k,(p)}$ by $E_pH_{0,(p)}H_{k,(p)}$ for $k=1,2,3,4$.

\noindent{\sc Step 2.} \em Global sum of squares\em. Define $U:=\bigcup_{p\in \{F=0\}\cap X}U^p$ and consider the invariant holomorphic function 
$$
B_k:U\to\C,\ z\mapsto B_k(z)= H_{k,(p)}\quad\text{if $z\in U^p$}
$$
for $k=1,2,3,4$. We have $G^2F=\sum_{k=1}^4B_k^2$ on $U$.

By Lemma \ref{extension} there exist invariant holomorphic functions $A_1,A_2,A_3,A_4\in\an(Y)$ such that $G^4F^2$ divides $A_k|_U-B_k$ for $k=1,2,3,4$. On the open set $U$ we obtain
$$
\sum_{k=1}^4A_k^2-G^2F=\sum_{k=1}^4A_k^2-\sum_{k=1}^4B_k^2=\sum_{k=1}^4(A_k^2-B_k^2).
$$
As $G^4F^2$ divides each $(A_k-B_k)(A_k+B_k)=A_k^2-B_k^2$ in $\an(U)$, it also divides $\sum_{k=1}^4A_k^2-G^2F$ in $\an(U)$. Thus, there exists an invariant holomorphic function $\Psi:U\to\C$ such that
$$
\sum_{k=1}^4A_k^2-G^2F=G^4F^2\Psi\quad\leadsto\quad\sum_{k=1}^4A_k^2=G^2F(1+\Psi G^2F)
$$
in $\an(U)$. As $1+\Psi G^2F$ does not vanish at any point $p\in\{F=0\}\cap X$, we may assume (after shrinking $U$ if necessary) that $1+\Psi G^2F$ is a holomorphic unit in $\an(U)$.

\noindent{\sc Step 3.} \em Additional square: There exists $\mu\in(0,+\infty)$ such that the function
$$
H_0:=\frac{\sum_{k=1}^4A_k^2+\mu^2G^4F^2}{G^2F}
$$
is holomorphic on an open neighborhood of $X$ in $Y$, its restriction $H_0|_{X}$ is strictly positive and $H_0|_{C_i}\neq0$ for each $i\geq1$\em. In particular, $G^2FH_0\in\Sos_5\an^\sigma(Y)^2$.

For each $\mu\in(0,+\infty)$ the zero-set of $\sum_{k=1}^4A_k^2+\mu^2G^4F^2$ in $X$ is contained in $\{F=0\}\cap X$. Thus, outside $\{F=0\}\cap X$, the restriction $H_0|_{X}$ is strictly positive. In addition, on $U\cap X$ 
$$
\frac{\sum_{k=1}^4A_k^2}{G^2F}+\mu^2G^2F=1+\Psi G^2F+\mu^2G^2F
$$
is a strictly positive analytic function, so $H_0|_X$ is a strictly positive analytic function. Pick $z_i\in C_i\setminus X$ for each $i\geq1$. As $G^2F$ does not vanish at any point of $C\setminus X$, there exists $\mu\in(0,+\infty)$ such that $H_0(z_i)\neq0$ for each $i\geq1$, so $H_0|_{C_i}\neq0$ for each $i\geq1$.
 
\noindent{\sc Step 4.} \em Eliminating extra zeros and conclusion\em. By Theorem \ref{maintool} applied to $H_0$ there exist, after shrinking $Y$ if necessary, invariant holomorphic functions $H,H_1,H_2:Y\to\C$ such that
\begin{itemize}
\item[(i)] $H_1,H_2\in\Sos_4\an^\sigma(Y)^2$.
\item[(ii)] $\{H=0\}\cap C\subset X\cap C$ and $\{H_\ell=0\}\cap X=\varnothing$ for $\ell=1,2$.
\item[(iii)] $H_0H_1=HH_2$ and $H_x\an_{Y,x}=H_{0,x}\an_{Y,x}$ for each $x\in X$.
\end{itemize}
By (ii) and (iii) we conclude that $H|_X$ is strictly positive and $\{H=0\}\cap(X\cup C)=\varnothing$, so, after shrinking $Y$ if necessary, we may assume $\{H=0\}=\varnothing$. Next, we have
$$
(GH_2)^2HF=HH_2G^2FH_2=(G^2FH_0)H_1H_2\in\Sos_8\an^\sigma(Y)^2
$$
and to finish it is enough to take $G':=GH_2$.
\end{proof}

\subsection{Analytic functions that are locally a square}
We prove next that a real analytic function on a $C$-analytic space of dimension $d$ that is locally a square can be represented as a sum of $d+1$ squares of analytic functions.

\begin{lem}\label{lsr}
Let $(X,\an_X)$ be a $C$-analytic space of dimension $d$ and $f\in\an(X)$ such that the analytic function germ $f_x$ is a square in $\an_{X,x}$ for each $x\in X$. Then there exist $g_1,\ldots,g_{d+1}\in\an(X)$ such that $f=g_1^2+\cdots+g_{d+1}^2$.
\end{lem}
\begin{proof}
For each $x\in X$ we pick a connected neighborhood $U\subset X$ such that $f|_U=g_U^2$ for some $g_U\in\an(U)$. We collect all the previous open sets to construct an open covering ${\mathfrak U}$ of $X$. If $y\in U_1\cap U_2$, the transition function $g_{12}:=g_{U_1}g_{U_2}^{-1}=\pm1$ is locally constant. Consider the coherent sheaf of ideals $\Ii_x:=g_{U,x}\an_{X,x}$ if $x\in U$. As $(X,\an_X)$ is a $C$-analytic space, it has an invariant Stein complexification $(Y,\an_Y)$ to which $\Ii$ extends as a coherent sheaf of ideals. As $X$ has dimension $d$, there exist by \cite{coen} $h_1,\ldots,h_{d+1}\in\an(X)$ such that $\Ii=(h_1,\cdots,h_{d+1})\an_X$. Thus, there exists a strictly positive analytic function $u\in\an(X)$ such that $f=u(h_1^2+\cdots+h_{d+1}^2)$. Let $v\in\an(X)$ be a strictly positive analytic function such that $u=v^2$. If we define $g_i:=vh_i$ for $i=1,\ldots,d+1$, we obtain $f=g_1^2+\cdots+g_{d+1}^2$, as required.
\end{proof}

Let $(Y,\an_Y)$ be a reduced complex analytic space. Denote the set of points $y\in Y$ at which $\an_{Y,y}$ is not a normal ring with $B(Y)$. By \cite[Ch.VI.Thm.5]{n1} the set $B(Y)$ is an analytic subset of $Y$. Let us show how we can construct an universal denominator for the coherent sheaf $\widetilde{\an}_Y$ of weakly holomorphic functions on $Y$. Assume in the following that $(Y,\an_Y)$ is a Stein space endowed with an anti-involution $\sigma:Y\to Y$ and let $(X,\an_X)$ be its real part space. Observe that $B(X)=B(Y)\cap X$ (see \S\ref{mainr}). Let $((\widehat{Y},\an_{\widehat{Y}}),\pi)$ be the normalization of $(Y,\an_Y)$ and recall that $\an_{\widehat{Y}}=\pi^*(\widetilde{\an}_Y)$ where $\widetilde{\an}_Y$ is the sheaf of weakly holomorphic functions on $Y$. Let $\widehat{\sigma}:\widehat{Y}\to\widehat{Y}$ be the anti-involution of $\widehat{Y}$ induced by $\sigma$, which satisfies $\pi\circ\widehat{\sigma}=\sigma\circ\pi$. The following result is inspired by \cite[E.73.a]{kk}.

\begin{lem}[Optimal universal denominator]\label{ud}
There exists an invariant holomorphic function $D\in\an(Y)$ such that $\{D|_X=0\}$ is contained in $X\cap B(Y)$, $D_y$ is not a zero divisor of $\an_{Y,y}$ and $D_yF_y\in\an_{Y,y}$ for each $F_y\in\widetilde{\an}_{Y,y}$ and each $y\in Y$.
\end{lem}
\begin{proof}
As $\an_Y$ and $\widetilde{\an}_Y$ are coherent sheaves of $\an_Y$-modules,
$$
\Jj_y:=(\an_{Y,y}:\widetilde{\an}_{Y,y}):=\{\xi_y\in\an_{Y,y}:\ \xi_y\widetilde{\an}_{Y,y}\subset\an_{Y,y}\}
$$
is a coherent sheaf of ideals of $\an_Y$. It holds that $\Jj_y=\an_{Y,y}$ if and only if $y\in Y\setminus B(Y)$, so $B(Y)$ is the support of the coherent sheaf of ideals $\an_Y/\Jj$. Pick points $y_m$ in each connected component $M_m$ of $Y\setminus\Sing(Y)$. As the family $\{Y_m\}_{m\geq1}$ is locally finite, $\{y_m\}_{m\geq1}$ is a discrete set. By \cite[\S52.5]{kk} there exists $D'\in H^0(Y,\Jj)$ such that $D'(y_m)=1$ and $D'(\sigma(y_m))=1$ for each $m\geq1$. By the Identity principle $D'_y(\ol{D'\circ\sigma})_y$ is a non-zero divisor of $\an_{Y,y}$ for each $y\in Y$. In addition, $D'_y(\ol{D'\circ\sigma})_yF_y\in\an_{Y,y}$ for each $F_y\in\widetilde{\an}_{Y,y}$ and each $y\in Y$. As $B(Y)$ is invariant, there exist (using inductively \cite[\S52.5]{kk}) finitely many holomorphic functions $D_i\in H^0(Y,\Jj)$ satisfying the same properties as $D'$ such that $B(Y)=\{D_1\ol{D_1\circ\sigma}=0,\ldots,D_r\ol{D_r\circ\sigma}=0\}$. Define $D:=\sum_{i=1}^rD_i\ol{D_i\circ\sigma}$ and observe that $D(y_m)=r\neq0$ for each $m\geq1$. The reader can check readily that $D$ satisfies the required conditions.
\end{proof}

A universal denominator satisfying the conditions of Lemma \ref{ud} will be called \em optimal universal denominator\em.

\begin{cor}\label{lsq}
Let $F\in\an(\widehat{Y})$ be an invariant holomorphic function such that $\{F=0\}=\varnothing$ and $D$ an optimal universal denominator. Then there exists an invariant holomorphic function $A\in\an(Y)$ such that $(D\circ\pi)^2F=A\circ\pi$ and $A_x$ is a square of $\an_{Y,x}$ for each $x\in X$.
\end{cor}
\begin{proof}
As $D$ is a universal denominator, $D^2$ is also a universal denominator, so there exists $A\in\an(\widehat{Y})$ such that $(D\circ\pi)^2F=A\circ\pi$. As $D,F$ are invariant holomorphic functions and $\sigma\circ\pi=\pi\circ\widehat{\sigma}$, $A$ is also invariant. Pick $x\in X$ and write $\pi^{-1}(x):=\{y_1,\ldots,y_r\}$. Let $W_i\subset\widehat{Y}$ be a neighborhood of $y_i$ such that $W_i\cap W_j=\varnothing$ if $i\neq j$ and there exists $B_i\in\an(W_i)$ such that $B_i^2=F|_{W_i}$ for $i=1,\ldots,r$. We may assume $W:=\bigcup_{i=1}^rW_i$ is invariant and define
$$
B:W\to\C,\ y\mapsto B_i(y)\ \text{ if $y\in W_i$,}
$$
which satisfies $B^2=F|_W$. As $\pi^{-1}(x)\subset W$ and $\pi$ is proper, we may assume $V:=\pi(W)$ is an open subset of $Y$ and $\pi^{-1}(V)=W$. As $D$ is a universal denominator, there exists $B'\in\an(V)$ such that $(D|_V\circ\pi)B=B'\circ\pi$, so
$$
(B'\circ\pi)^2=((D|_V\circ\pi)B)^2=(D|_V\circ\pi)^2F|_W=A\circ\pi.
$$
As $\pi^*:\an(V)\to\an(W),\ G\mapsto G\circ\pi$ is injective, we conclude $A|_V=B'^2$, so $A_x$ is locally a square of $\an_{Y,x}$, as required.
\end{proof}

\subsection{Proof of Theorem \ref{h17}}
We distinguish two cases:

\noindent{\sc General case:} By \S\ref{complexification} there exists a reduced Stein space $(Y,\an_Y)$ endowed with an anti-involution $\sigma$ such that $(X,\an_X)$ is the real part space of $(Y,\an_Y)$ and $(Y,\an_Y)$ is a complexification of $(X,\an_X)$. By \S\ref{norm} there exists a normalization $((\widehat{Y},\an_{\widehat{Y}}),\pi)$ of $(Y,\an_Y)$ and an anti-involution $\widehat{\sigma}$ of $\widehat{Y}$ such that $\sigma\circ\pi=\pi\circ\widehat{\sigma}$. By \cite{n3} $(\widehat{Y},\an_{\widehat{Y}})$ is a Stein space. Let 
$$
\widehat{X}=\{z\in Y:\ \widehat{\sigma}(z)=z\}.
$$ 
By Corollary \ref{dim12} there exists a ($\widehat{\sigma}$-)invariant analytic curve $C\subset\widehat{Y}$ of real dimension $1$ without isolated points such that $D_0:=C\cap\widehat{X}$ is a discrete subset of $\widehat{X}$ and $\pi^{-1}(X)=\widehat{X}\cup C$. As $X$ has a system of open Stein neighborhoods in $Y$, by \cite[M.Thm.3]{gu} also $\widehat{X}\cup C$ has a system of open Stein neighborhoods in $\widehat{Y}$. As the irreducible components of $\widehat{Y}$ are its connected components, we may assume (after shrinking $\widehat{Y}$ if necessary) that both $\widehat{X}\cup C$ and $\widehat{Y}$ are connected \cite[Thm.1.2, Prop.5.16]{fe12}.

Let $f:X\to\R$ be a non-zero positive semidefinite analytic function. Shrinking $Y$ if necessary, we may assume that $f$ extends to an invariant holomorphic function $F:Y\to\C$. Consider the holomorphic function $F':=F\circ\pi:\widehat{Y}\to\C$. As $\sigma\circ\pi=\pi\circ\widehat{\sigma}$, we have that $F':=F\circ\pi$ is an invariant holomorphic function whose restriction to $\widehat{X}$ is positive semidefinite. Observe that $\{F'=0\}$ has dimension $1$.

By Lemma \ref{dim1} there are invariant holomorphic functions $G_1,F_1,F_2\in\an(\widehat{Y})$ such that $\{F_i=0\}\cap\widehat{X}$ is a discrete set contained in $\{F'=0\}\cap\widehat{X}$, $F_i|_{\widehat{X}}\geq0$ and $F_1^2F'=(G_1^2+F'^2)F_2$. By Theorem \ref{movc} there exists, after shrinking $\widehat{Y}$ if necessary, an invariant holomorphic function $G_2\in\an(\widehat{Y})$ such that the restriction of $F_2':=F_2-G_2^2$ to $\widehat{X}$ is positive semidefinite, $\{F_2'=0\}\cap(\widehat{X}\cup C)$ is a discrete set and $\{F_2'=0\}\cap \widehat{X}=\{F_2=0\}\cap\widehat{X}$. By Theorem \ref{maintool} there exist, after shrinking $\widehat{Y}$ if necessary, invariant holomorphic functions $F_2'',H_1,H_2\in\an(\widehat{Y})$ such that:
\begin{itemize}
\item[(i)] $H_1,H_2\in\Sos_4\an^\sigma(\widehat{Y})^2$.
\item[(ii)] $\{F_2''=0\}\cap C\subset\widehat{X}\cap C$ and $\{H_\ell=0\}\cap\widehat{X}=\varnothing$ for $\ell=1,2$.
\item[(iii)] $F_2'H_1=F_2''H_2$ and $F_{2,x}''\an_{\widehat{Y},x}=F_{2,x}'\an_{\widehat{Y},x}$ for each $x\in\widehat{X}$. In particular, $\{F_2'=0\}\cap\widehat{X}=\{F_2''=0\}\cap\widehat{X}$.
\end{itemize}
By Theorem \ref{discretezero-set} there are, after shrinking $\widehat{Y}$ if necessary, invariant holomorphic functions $H_3,G_3\in\an(\widehat{Y})$ such that $\{H_3=0\}=\varnothing$, ($H_3|_{\widehat{X}}$ is strictly positive), 
$$
\{G_3=0\}\cap\widehat{X}\subset\{F_2''=0\}\cap\widehat{X}=\{F_2'=0\}\cap\widehat{X}=\{F_2=0\}\cap\widehat{X}\subset\{F'=0\}\cap\widehat{X}
$$ 
is a discrete set and $G_3^2H_3F_2''\in\Sos_8\an^\sigma(\widehat{Y})^2$.

Let $D\in\an(Y)$ be an optimal universal denominator. By Corollary \ref{lsq} we find an invariant holomorphic function $A\in\an(Y)$ such that $(D\circ\pi)^2H_3=A\circ\pi$ and $A_x$ is a square in $\an_{Y,x}$ for each $x\in X$. By Lemma \ref{lsr} we have, after shrinking $Y$ if necessary, invariant holomorphic functions $A_1,A_2,A_3\in\an(Y)$ such that $A=A_1^2+A_2^2+A_3^2$. Thus, if we write $D':=D\circ\pi$,
$$
D'^2H_3=(A_1\circ\pi)^2+(A_2\circ\pi)^2+(A_3\circ\pi)^2\in\Sos_3\an^\sigma(\widehat{Y})^2.
$$
Consequently, $(G_3D'H_3)^2F_2''\in\Sos_8\an^\sigma(\widehat{Y})^2$. We deduce
$$
(G_3D'H_3H_1)^2(F_2-G_2^2)=(G_3D'H_3H_1)^2F_2'=(G_3D'H_3)^2H_1H_2F_2''\in\Sos_8\an^\sigma(\widehat{Y})^2,
$$
so $(G_3D'H_3H_1)^2F_2\in\Sos_9\an^\sigma(\widehat{Y})^2$. Thus,
$$
(F_1G_3D'H_3H_1)^2F'=(G_1^2+F'^2)(G_3D'H_3H_1)^2F_2\in\Sos_{10}\an^\sigma(\widehat{Y})^2.
$$
As $\{F_1G_3H_3H_1=0\}\cap\widehat{X}\subset\{F'=0\}\cap\widehat{X}$ is a discrete set, $F_1G_3H_3H_1$ is a non-zero divisor of $\an(\widehat{Y})$. Thus, using that $D'$ is an optimal universal denominator, we find $E\in\an(Y)$ such that $D'(F_1G_3H_3H_1)=E\circ\pi$, $E$ is not a zero divisor and $\{E=0\}\cap X\subset B(X)\cup(\{F=0\}\cap X)$. As $D'^2\Sos_{10}\an^\sigma(\widehat{Y})^2\subset\pi^*(\Sos_{10}(\an^\sigma(Y))^2)$, we deduce that $E^2F\in\Sos_{10}\an^\sigma(Y)^2$ and $\{E=0\}\cap X\subset B(X)\cup(\{F=0\}\cap X)$. This concludes the general case.

\noindent{\sc Coherent case:} Let us modify the previous proof when $X$ is coherent. If such is the case, $\pi^{-1}(X)=\widehat{X}$ and $C=\varnothing$ (Remark \ref{casd2}(ii)). Thus, we can work directly on $\widehat{X}$ and do not need to care about $\widehat{Y}$. This simplifies everything and only $5$ squares will be enough. Denote $\rho:=\pi|_{\widehat{X}}:\widehat{X}\to X$. As the irreducible components of $\widehat{X}$ are its connected components, we assume $\widehat{X}$ is connected.

Let $f\in\an(X)$ be a non-zero positive definite analytic function and $f':=f\circ\rho$. Observe that $\{f'=0\}$ has dimension $\leq1$. By Lemma \ref{dim1} there exist analytic functions $g_1,f_1,f_2\in\an(\widehat{X})$ such that each $f_i$ is positive semidefinite, has a discrete zero-set contained in $\{f'=0\}$ and $f_1^2f'=(g_1^2+f'^2)f_2$. By {\sc Step 2} of the proof of Theorem \ref{discretezero-set} there exist analytic functions $g_2,a_1,a_2,a_3,a_4\in\an(\widehat{X})$ such that $\{g_2=0\}\subset\{f_2=0\}$ and $g_2^2f_2\an_{\widehat{X},x}=(a_1^2+a_2^2+a_3^2+a_4^2)\an_{\widehat{X},x}$ for each $x\in\{f_2=0\}$. We still need to modify the analytic functions $a_i$. Let $\{Z_i\}_{i\geq1}$ be the irreducible components of dimension $1$ of $\{f'=0\}$ and pick $z_i\in Z_i\setminus\{g_2^2f_2=0\}$ for each $i\geq1$. Let $\lambda\in(0,+\infty)$ be such that $a_1':=a_1+\lambda g_2^2f_2$ does not vanish at any of the points $z_i$. The set $\{f'=0,a_1'=0\}$ is discrete, so $D:=\{f'=0,a_1'=0,a_2\neq0\}$ is also discrete. Let $c\in\an(\widehat{X})$ be an analytic function such that $\{c=0\}=D$ and define $a_2':=a_2+cg_2^2f_2$, $a_3':=a_3$ and $a_4':=a_4$. Observe that $g_2^2f_2\an_{\widehat{X},x}=(a_1'^2+a_2'^2+a_3'^2+a_4'^2)\an_{\widehat{X},x}$ for each $x\in\{f_2=0\}$ and 
\begin{equation*}
\begin{split}
\{a_1'=0,a_2+cg_2^2f_2&=0,f'=0,a_2\neq0\}\\
&=\{a_1'=0,f'=0,a_2\neq0\}\cap\{a_2+cg_2^2f_2=0\}\\
&=\{c=0\}\cap\{a_2+cg_2^2f_2=0\}=\{a_2=0,c=0\}=\varnothing.
\end{split}
\end{equation*}
Consequently, as $\{c=0,a_2=0\}=\varnothing$, we obtain
\begin{equation*}
\begin{split}
\{a_1'^2+a_2'^2+a_3'^2+a_4'^2=0\}&\cap\{f'=0,f_2\neq0\}
\subset\{a_1'=0,a_2'=0,f'=0,f_2\neq0\}\\
&=\{a_1'=0,a_2+cg_2^2f_2=0,f'=0,f_2\neq0\}\\
&=\{a_1'=0,a_2+cg_2^2f_2=0,f'=0,f_2\neq0,a_2\neq0\}
=\varnothing.
\end{split}
\end{equation*}
This means that $g_2^2f_2\an_{\widehat{X},x}=(a_1'^2+a_2'^2+a_3'^2+a_4'^2)\an_{\widehat{X},x}$ for each $x\in\{f'=0\}$, so
$$
(g_2f_1)^2f'\an_{\widehat{X},x}=(g_1^2+f'^2)g_2^2f_2\an_{\widehat{X},x}=(g_1^2+f'^2)(a_1'^2+a_2'^2+a_3'^2+a_4'^2)\an_{\widehat{X},x}
$$
for each $x\in\{f'=0\}$. Let $a_1'',a_2'',a_3'',a_4''\in\an(\widehat{X})$ be such that 
$$
(g_1^2+f'^2)(a_1^2+a_2^2+a_3^2+a_4^2)=a_1''^2+a_2''^2+a_3''^2+a_4''^2.
$$
As $\{g_2f_1=0\}\subset\{f'=0\}$,
$$
u:=\frac{(g_2f_1)^2f'}{a_1''^2+a_2''^2+a_3''^2+a_4''^2+((g_2f_1)^2f')^2}\in\an(\widehat{X})
$$
is a strictly positive analytic function. Let $v\in\an(\widehat{X})$ be a strictly positive analytic function such that $u=v^2$. Thus,
$$
(g_2f_1)^2f'=(a_1''v)^2+(a_2''v)^2+(a_3''v)^2+(a_4''v)^2+((g_2f_1)^2f'v)^2.
$$
Let $d\in\an(X)$ be an optimal universal denominator, which satisfies $\{d=0\}=B(X)$. Let $b_0,b_1,b_2,b_3,b_4,b_5\in\an(X)$ be such that 
$$
b_k\circ\rho=\begin{cases}
(d\circ\pi)g_2f_1&\text{if $k=0$,}\\
(d\circ\pi)a_k''v&\text{if $k=1,2,3,4$,}\\
(d\circ\pi)(g_2f_1)^2f'v&\text{if $k=5$.}
\end{cases}
$$
We obtain $b_0^2f=\sum_{k=1}^5b_k^2$ and $\{b_0=0\}\subset\{f=0\}\cup\{d=0\}=\{f=0\}\cup B(X)$, as required.
\qed

\renewcommand\refname

\end{document}